\documentclass[12pt]{article}

\usepackage{amsmath,amssymb,amsthm}

\usepackage{verbatim}

\newcommand{\R}{\mathbb{R}}

\newcommand{\N}{\mathbb{N}}
\newcommand{\C}{\mathbb{C}}

\newcommand{\Hy}{\mathbb{H}}
\newcommand{\Hc}{\mathcal{H}}

\def\XXint#1#2#3{{\setbox0=\hbox{$#1{#2#3}{\int}$ }
\vcenter{\hbox{$#2#3$ }}\kern-.6\wd0}}

\author{Brian Allen}

\title{ODE Maximum Principle at Infinity and Non-Compact Solutions of IMCF in Hyperbolic Space}

\newtheorem{Thm}{Theorem}
\newtheorem{Cor}{Corollary}
\newtheorem{Prop}{Proposition}
\newtheorem{Lem}{Lemma}

\newtheorem{Def}{Definition}
\newtheorem{Ex}{Example}

\begin{document}

\maketitle

\begin{abstract}
In this work we extend the ODE Maximum principle of Hamilton \cite{Ha} to non-compact hypersurfaces using the Omari-Yau maximum principle at infinity \cite{CY,O,P,Y}. As an application of this result, we investigate Inverse Mean Curvature Flow (IMCF) of non-compact hypersurfaces in hyperbolic space. Specifically, we look at bounded graphs over horospheres in $\Hy^{n+1}$ and show long time existence of the flow as well as asymptotic convergence to horospheres. 
\end{abstract}

\section{Introduction}
\label{intro}

Non-compact maximum principles are important to the study of non-compact solutions of geometric evolution equations where standard maximum principles do not apply. Using a maximum principle which follows from Huisken's monotonicty formula, Ecker and Huisken \cite{EH1} were able to show convergence under MCF to a translating soliton for graphs over planes in $\R^{n+1}$, satisfying certain initial growth conditions. Later, they developed further interior estimates for non-compact MCF \cite{EH2} as well as a non-compact maximum principle that works for a fairly general class of evolution equations with time dependent metrics including Ricci Flow.

The non-compact maximum principles mentioned above follow the standard parabolic PDE techniques where evolution equations need to be controlled on the whole domain of definition in order for the maximum principle to apply. In the literature on MCF and IMCF, though, there have been examples of cases where evolution equations cannot be controlled on the whole domain of definition but where the specific geometry around a max or min can be exploited to control the equation at these points. This is where an ODE maximum principle, such as Hamilton's maximum principle \cite{Ha,CM}, is most valuable and why the ODE maximum principle at infinity, a non-compact version of Hamilton's work, is important to the study of non-compact evolution equations (See Theorem \eqref{barriers} for an illustrative example of this phenomenon).

To illustrate the importance of the ODE maximum principle at infinity we will apply it to the geometric evolution of hypersurfaces $\Sigma^n$ through a one parameter family of embeddings $\varphi: \Sigma \times [0,T) \rightarrow \Hy^{n+1}$, $\varphi$ satisfying inverse mean curvature flow

\begin{equation}\label{IMCF}
\begin{cases}
\frac{\partial \varphi}{\partial t}(p,t) = \frac{\nu(p,t)}{H(p,t)}  &\text{ for } (p,t) \in \Sigma \times [0,T)
\\ F(p,0) = \Sigma_0  &\text{ for } p \in \Sigma 
\end{cases}
\end{equation}
where $H$ is the mean curvature of $\Sigma_t := \varphi_t(\Sigma)$ and $\nu$ is a consistently chosen normal vector (we will be more specific later).

 Global existence results for initial hypersurfaces in euclidean space were first obtained by Gerhardt \cite{CG1} and Urbas \cite{U}. They independently proved that any compact, mean-convex and star-shaped hypersurface will asymptotically approach a sphere and converge to a sphere after an appropriate rescaling under IMCF (as well as a whole family of inverse flows). 

Since then there have been extensions of this theorem to Lorentzian manifolds \cite{CG2}, hyperbolic space \cite{CG3,D} as well as to rotationally symmetric spaces with non-positive radial curvature \cite{S}. There has also been a great deal of work on weak solutions of IMCF including viscosity solutions \cite{CG}, weak solutions through connection to the p-Laplacian \cite{M} as well as the most famous formulation of weak variational solutions to IMCF by Huisken and Ilmanen \cite{HI2} which were used to prove the Riemannian Penrose Inequality (time symmetric case). 

The non-compact case of IMCF has seen almost no attention besides the specific examples given by Huisken and Illmanen \cite{HI1} and the recent papers on solitons of IMCF by Drugan, Lee and Wheeler \cite{DLW}, Drugan, Fong and Lee \cite{DLW2}, and Castro and Lerma \cite{CL} . Besides these examples of special solutions there has been no work on showing convergence to a prototypical hypersurface for a class of initial data as has been done for compact IMCF for the sphere. The present work changes this by  applying the ODE maximum principle at infinity to the study of non-compact IMCF in Hyperbolic space and more precisely we prove the following theorem.

\begin{Thm} \label{LTE}
Let $\Sigma_t$ be a smooth solution of IMCF with initial hypersurface $\Sigma_0$ satisfying the following bounds on the mean curvature and second fundamental form, $0 < H_0 \le H(x,0) \le H_1 < \infty$ and $|A|(x,0) \le A_0 < \infty$. We further assume that $\Sigma_0$ can be represented as a graph of a bounded function with bounded gradient, over and uniformly bounded away from $\R^n \times \{0\}$ in the upper half space  model of hyperbolic space. Then the IMCF starting at $\Sigma_0$ exists for all time $t \in [0,\infty)$ and the solution asymptotically converges to a horosphere.
\end{Thm}
\vspace{0.25cm}

In the second section we state and prove an ODE maximum principle at infinity which allows us to use the Omari-Yau maximum principle at infinity  \cite{CY,O,P,Y} to extend the ODE maximum principle of Hamilton \cite{Ha}, \cite{CM} to the case of bounded (in space) functions defined on non-compact domains. 

In the third section we use the ODE maximum principle at infinity to prove Theorem \eqref{LTE}. We start by showing short time existence to the flow, move on to long time existence and finish with asymptotic properties. In this section we highlight the usefulness of the ODE maximum principle at infinity in Theorem \ref{barriers} and we point out that some different methods are used in Lemma \ref{C^2Decay} to find $C^2$ decay for the flow.

%\begin{acknowledgements}
%\end{acknowledgements}

\section{ODE Maximum Principle at Infinity}
\label{sec:1}

In this section we state and prove an ODE maximum principle that works for functions defined on non-compact domains and will be applied to study non-compact solutions of IMCF in Hyperbolic space in the next section. This is an extension of the work of Hamilton \cite{Ha} which is described in detail in \cite{CM}.

\begin{Thm}\label{ODEmax}
Assume for $t \in [0,T)$ that $g(t)$ is a family of Riemannian metrics defined on the manifold $M^n$ so that the dependence on $t$ is smooth. We also assume that $g_t$ is a metric to which the Omori-Yau maximum principle at infinity (See Theorem \ref{OYMax}) applies for each $t \in [0,T)$.

Let $u : M \times [0,T) \rightarrow \R$ be a smooth function so that  $|u(x,t)| \le C(t)$, $|\nabla u| \le C(t)$ and $|u_t|_{C^{\alpha}}\le C(t)$ for each time $t \in (0,T)$, satisfying
\begin{align}
\left (\partial_t - H^{ij}\nabla_i^{g_t}\nabla_j^{g_t}\right ) u = \langle X(x,u,\nabla^{g_t} u,t), \nabla^{g_t}u \rangle_{g_t} + F(u)
\end{align}
where $|X| \le C_1(t)$, $F$ is a locally Lipschitz function on $\R$ and $H_{ij}$ is a symmetric, positive definite matrix so that $|H| \le C_0$.

Setting $u_{sup}(t) = \sup_{x \in M}u(x,t)$ we have that the function, $u_{sup(t)}$ is locally Lipschitz and hence differentiable at almost every time $t \in [0,T)$. At every differentiable time we have that
\begin{align}
\frac{d u_{sup}(t)}{dt} = \lim_{k \rightarrow \infty} \frac{\partial u(x_k,t)}{\partial t} \hspace{1 cm} &\text{ where } \{x_k\} \subset \R^n \text{ is any sequence such}
\\& \text{that } \lim_{k \rightarrow \infty} u(x_k, t) = \sup_{x \in \R^n}u(x,t)
\end{align}

If $\varphi : [0,T') \rightarrow \R$ is a maximal solution of the ODE
\begin{align}
\begin{cases}
\varphi'(t) &= F(\varphi(t))
\\  \varphi(0) &= u_{sup}(0)
\end{cases}
\end{align}
 then we have that $u(x,t) \le \varphi(t)$ for $(x,t) \in M \times [0,\min\{T,T'\})$.
\end{Thm}

\textbf{Note:} We did not impose conditions that imply the Omori-Yau maximum principle (result which gives a sign on the Hessian or the Laplacian of a function at max/mins if appropriate curvature bounds are satisfied) since there are fairly general assumptions that may be useful depending on the application. What is important is that you can show that some hypotheses that guarantee the Omori-Yau maximum principle are in place in order to apply Theorem \ref{ODEmax}. With respect to our application of Theorem \ref{ODEmax} to IMCF in Hyperbolic space, we will ensure that the hypotheses of Theorem \ref{OYMax} are satisfied in order to apply Theorem \ref{ODEmax}.

Before we can prove this theorem we will need the following lemma.

\begin{Lem} \label{ader}
Let $u:M^n \times (0,T) \rightarrow \R$ be a bounded $C^1$ function so that $|u(x,t)| \le C(t)$, $|\nabla u| \le C(t)$ and $|u_t|_{C^{\alpha}}\le C(t)$ then $u_{sup}: (0,T) \rightarrow \R$, defined as $u_{sup}(t) = \sup_{x \in M}u(x,t)$, is a locally Lipschitz function in $(0,T)$. Also, at every differentiable time $t \in (0,T)$ we have that

\begin{align}
\frac{d u_{sup}(t)}{dt} = \lim_{k \rightarrow \infty} \frac{\partial u(x_k,t)}{\partial t} \hspace{1 cm} &\text{ where } \{x_k\} \subset M \text{ is any sequence such}
\\&\text{that } \lim_{k \rightarrow \infty} u(x_k, t) = \sup_{x \in M}u(x,t)
\end{align}
\end{Lem}

\textbf{Note:} If $u$ attains its max at some point $x \in M$ then we can take the trivial sequence which is constantly equal to $x$.

\textbf{Note:} This is an extension of Hamilton's work \cite{Ha}, explained in \cite{M}, to non-compact manifolds where we allow $\sup$ and $\inf$ instead of just $\max$ and $\min$ on compact manifolds.

\begin{proof}
Fix a $ t \in (0,T)$ and then choose a $\delta > 0$ so that $[t - \delta, t + \delta] \subset (0,T)$. Then choose an $\epsilon$ so that $0 < \epsilon < \delta$ and note that since $u$ is bounded and $C^1$ on $M \times (0,T)$ we know that for every $x \in M$, there exists some Lipschitz constant $K > 0$, depending on $t$ and $\epsilon$, so that $u(x,t+ \epsilon) - u(x,t) \le K \epsilon$. Note that the constant $K$ is uniform in $x$ by the assumption $|\nabla u| \le C(t)$.

Now for each $\epsilon > 0$ we can find a sequence $\{x_k^{\epsilon}\}$ so that $\displaystyle u_{sup}(t + \epsilon) = \lim_{k \rightarrow \infty} u(x_k^{\epsilon},t+ \epsilon)$ and hence
\begin{align}
u_{sup}(t + \epsilon) &= \lim_{k \rightarrow \infty} u(x_k^{\epsilon},t+ \epsilon) \le \limsup_{k \rightarrow \infty} u(x_k^{\epsilon},t) + K\epsilon 
\\&\le \lim_{k \rightarrow \infty}u(x_k^0,t) + K \epsilon = u_{sup}(t) + K \epsilon
\end{align}
where the second inequality follows from the fact that $\displaystyle u_{sup}(t) = \lim_{k \rightarrow \infty} u(x_k^0,t)$. So we have found that $u_{sup}(t + \epsilon) - u_{sup}(t) \le K \epsilon$. Repeating this argument for $- \delta < \epsilon < 0$ we conclude that $u_{sup}$ is a locally Lipschitz function on $(0,T)$ and hence differentiable at almost every time $t$.

Let $t \in (0,T)$ be a time where $u_{sup}$ is differentiable and let $\{x_k\}$ be a sequence so that $\displaystyle \lim_{k \rightarrow \infty} u(x_k,t) = \sup_{x \in M}u(x,t)$. Then by the Mean Value Theorem, for every $0 < \epsilon < \delta$ we can choose a $s_k^{\epsilon} \in (t,t +\epsilon)$ so that $u(x_k,t + \epsilon) = u(x_k,t) + \epsilon \frac{\partial u(x_k,s_k^{\epsilon})}{\partial t}$ and so
\begin{align}
u_{sup}(t + \epsilon) &\ge \limsup_{k \rightarrow \infty} u(x_k, t + \epsilon) = \limsup_{k \rightarrow \infty} \left [u(x_k,t) + \epsilon \frac{\partial u(x_k,s_k^{\epsilon})}{\partial t} \right ]
\\&= u_{sup}(t) +\epsilon \limsup_{k \rightarrow \infty}\frac{\partial u(x_k,s_k^{\epsilon})}{\partial t}
\end{align}
so then by rearranging we find
\begin{align}
\frac{u_{sup}(t + \epsilon) - u_{sup}(t)}{\epsilon} \ge \limsup_{k \rightarrow \infty}\frac{\partial u(x_k,s_k^{\epsilon})}{\partial t}
\end{align}
and so by letting $\epsilon \rightarrow 0$ we find that 
\begin{align}
\displaystyle \frac{d u_{sup}(t)}{dt}&=\lim_{\epsilon \rightarrow 0} \frac{u_{sup}(t + \epsilon) - u_{sup}(t)}{\epsilon} 
\\&\ge \lim_{\epsilon \rightarrow 0} \limsup_{k \rightarrow \infty}\frac{\partial u(x_k,s_k^{\epsilon})}{\partial t}\ge \limsup_{k \rightarrow \infty} \frac{\partial u(x_k,t)}{\partial t}\label{switchingLimits}
\end{align}
where we are justified in switching the limits in \eqref{switchingLimits} because of the uniformity that the assumption $|u_t|_{C^{\alpha}}\le C(t)$ provides.

Now if we repeat this argument for $- \delta < -\epsilon < 0$  we will get the following
\begin{align}
\frac{d u_{sup}(t)}{dt} \le \liminf_{k \rightarrow \infty} \frac{\partial u(x_k,t)}{\partial t}
\end{align}

Putting this all together we see that
\begin{align}
\limsup_{k \rightarrow \infty} \frac{\partial u(x_k,t)}{\partial t} \le \frac{d u_{sup}(t)}{dt} \le \liminf_{k \rightarrow \infty} \frac{\partial u(x_k,t)}{\partial t}
\end{align}
which tells us that $\displaystyle \lim_{k \rightarrow \infty} \frac{\partial u(x_k,t)}{\partial t}$ must converge at a differentiable time of $u_{sup}(t)$ and equal its derivative.
\end{proof}

\begin{proof} 
By the previous Lemma we know that $u_{sup}(t)$ is locally Lipschitz and hence differentiable almost everywhere in $[0,T)$. If we let $t \in [0,T)$ be a differentiable time and $\{x_k\}$ a sequence so that $\displaystyle \lim_{k\rightarrow \infty}u(x_k,t) = \sup_{x\in M}u(x,t)$, $|\nabla u(x_k,t)| < \frac{1}{k}$ and $\nabla_i\nabla_j u(x_k,t) < \frac{1}{k}g_{ij}$, which is guaranteed by the maximum principle at infinity, then we find

\begin{align}
&\frac{du_{sup}}{dt}(t) = \lim_{k \rightarrow \infty} \frac{\partial u}{\partial t}(x_k,t) 
\\&\le \limsup_{k \rightarrow \infty} \left (H^{ij} \nabla_i\nabla_j u(x_k,t) + \langle X(x_k,u,\nabla u,t), \nabla u(x_k,t)\rangle + F(u(x_k,t)) \right )
\\&\le \limsup_{k \rightarrow \infty} \left ( \frac{nC_0}{k} + \frac{|X|}{k} + F(u(x_k,t))\right )
\\&\le F\left (\limsup_{k \rightarrow \infty}u(x_k,t)\right) = F(u_{sup}(t))
\end{align}
and so we have that, at a differentiable time $t$
\begin{align}
\frac{du_{sup}}{dt}(t)\le  F(u_{sup}(t))
\end{align}

At this point we follow the argument from \cite{CM}. Now let $\varphi :[0,T') \rightarrow \R$ be as in the statement of the Theorem and for $\epsilon > 0$ let $\varphi_{\epsilon}:[0,T_{\epsilon}) \rightarrow \R$ be the maximal solution of the family of ODEs
\begin{align}
\begin{cases}
\varphi_{\epsilon}'(t) &= F(\varphi_{\epsilon}(t))
\\ \varphi_{\epsilon}(0) &= u_{sup}(0) + \epsilon
\end{cases}
\end{align}

Since $F$ is Lipschitz on compact sets we can restrict ourselves to $[0,T_{\delta}]$ for $T_{\delta} < \min\{T,T',T_0\}$ where we know that $u$ and $\varphi_{\epsilon}$ are bounded, for small enough $\epsilon$, and hence solutions to the above ODE have continuous dependence on the initial conditions (over compact time intervals). Hence using the fact that the family of functions $\varphi_{\epsilon}$ is uniformly Lipschitz for small enough $\epsilon$ we find that $\varphi_{\epsilon} \rightarrow \varphi$ uniformly on $[0,T_{\delta}]$ for any $T_{\delta}<\min\{T,T',T_0\}$ as $\epsilon \rightarrow 0$.

Now fix $\epsilon > 0$ and for sake of contradiction assume that there is some positive time so that $u_{sup}(t) > \varphi_{\epsilon}(t)$ and let $\bar{t}>0$ be the infimum of all such times which we know is $\not = 0$ since $u_{sup}(0) = \varphi_{\epsilon}(0) - \epsilon$. So $u_{sup}(\bar{t}) = \varphi_{\epsilon}(\bar{t})$ and hence we can let $\Phi_{\epsilon}(t) = \varphi_{\epsilon}(t) - u_{sup}(t)$. Then at differentiable times for $u_{sup}(t)$ in the interval $[0,\bar{t})$ we know that $\Phi_{\epsilon}(t) > 0$ and 

\begin{align}
\Phi_{\epsilon}'(t) \ge F(\varphi_{\epsilon}(t)) - F(u_{sup}(t)) \ge - C_{\epsilon}(\varphi_{\epsilon}(t) - u_{sup}(t)) = -C_{\epsilon}\Phi_{\epsilon}(t)
\end{align}
where $C_{\epsilon}$ is a local Lipschitz constant for $F$ in the interval $\{\varphi_{\epsilon}(s) : 0\le s \le \bar{t}\}$ and this differential inequality hold for a.e. $t \in [0,\bar{t}]$. 

Then by integrating this equation we find that $\Phi_{\epsilon}(t) \ge \Phi_{\epsilon}(0) e^{-C_{\epsilon}t} = \epsilon e^{-C_{\epsilon}t}$ and so in particular $\Phi_{\epsilon}(\bar{t}) \ge \epsilon e^{-C_{\epsilon}\bar{t}} > 0$ but that contradicts the fact that $\Phi_{\epsilon}(\bar{t}) = 0$.

So $u_{sup}(t) \le \varphi_{\epsilon}(t)$ for every $t \in [0,T_{\delta})$ and so if we let $\epsilon \rightarrow 0$ then we have that $u_{sup}(t) \le \varphi(t)$ for every $t \in [0,T_{\delta})$. Since $\delta > 0$ was arbitrary, we have proven the desired result for $[0,\min\{T,T'\})$.
\end{proof}

\section{Non-Compact Solutions to IMCF in Hyperbolic Space}
\label{sec:2}

In this section we apply the ODE maximum principle at infinity to the study of non-compact solutions of IMCF in $\Hy^{n+1}$. Our aim is to highlight the differences from the compact case of IMCF but we do not intend to include all of the standard details. Therefore, once the usefulness of the ODE maximum principle has been demonstrated and the different details that show up in this case are illustrated we will point to standard references to finish the proof of Theorem \eqref{LTE}. For detailed computations of all the evolution equations used in this paper as well as a thorough treatment of short time existence, similar to  what is done in \cite{CG2}, see my dissertation \cite{BA} .

It is convenient for us to use the upper half space model of $\Hy^{n+1}$ which is defined on the space $\R^{n+1}_+ = \R^n \times (0,\infty)$ with coordinates $(x_1,...,x_n,y)$ and the following metric

\begin{align}
\bar{g}=\frac{1}{y^2} \left ( dx_1^2 + ...+dx_n^2 + dy^2\right )
\end{align} 
where we denote the coordinate basis vectors as $\partial_{x_1},...,\partial_{x_n}, \partial_{y}=\partial_{x_{n+1}}$. In particular we will be looking at solutions which can be written as graphs over $\R^n \times \{0\}$, i.e. if $y(x,t): \R^n\times [0,\infty) \rightarrow \R$ then $\varphi(x,t) = (x, y(x,t))$ and $\Sigma_t = \varphi(\R^n\times\{t\})$. Then we can define $v = \sqrt{1 + |\nabla^0y|^2}$ where $\nabla^0$ denotes derivatives w.r.t. the flat metric on $\R^{n+1}$. It will also be useful to define $w = \bar{g}(\nu, \eta)= \frac{1}{vy}$, where $\eta = -\partial_y$  and $\nu$ is the downward pointing normal (the downward pointing normal makes IMCF forwards parabolic). We will use $\nabla$ and $\langle \cdot,\cdot\rangle$ for the connection and metric with respect to $\Sigma_t$.

\subsection{Short Time Existence}\label{Subsec:STE}
 The goal of this subsection is to prove short time existence to IMCF in Theorem \eqref{STE} for the case of non-compact graphs over the plane $\{y=0\}$, satisfying the conditions of Theorem \ref{LTE}. We mostly follow the proof of short time existence given in \cite{CG2} where Gerhardt shows short time existence in the compact case. We use standard parabolic Holder spaces which we now define for the convenience of the reader.
 
 \begin{Def}
In $\R^n \times [0,T)$ we defined the parabolic distance between $p_1=(\textbf{x}_1,t_1)$ and $p_2=(\textbf{x}_2,t_2)$ as
\begin{align}
\rho(p_1,p_2) = |\textbf{x}_1-\textbf{x}_2| + |t_1-t_2|^{1/2}
\end{align}
\end{Def}

\begin{Def}
For $u: \R^n \times [0,T) \rightarrow \R$, $\alpha \in (0,1)$ we define
\begin{align}
[u]_{\alpha/2,\alpha} &= \sup_{p_1\not = p_2} \frac{|u(p_1) - u(p_2)|}{\rho(p_1,p_2)^{\alpha}}
\\ |u|_0 &= \sup_{\R^n \times [0,T)} |u|
\\|u|_{\alpha/2,\alpha} &= |u|_0 + [u]_{\alpha/2,\alpha}
\end{align}
\end{Def}

\begin{Def}
We define $C^{\alpha/2,\alpha}(\R^n\times [0,T))$ as the set of all functions $u$ so that $ |u|_{\alpha/2,\alpha} < \infty$. Also, we define $C^{1 + \alpha/2,2+\alpha}$ as the set of all functions $u$ so that
\begin{align}
[u]_{1 + \alpha/2,2 + \alpha} := [u_t]_{\alpha/2,\alpha} + \sum_{i,j=1}^n [u_{x_ix_j}]_{\alpha/2,\alpha} < \infty
\end{align}
and
\begin{align}
|u|_{1 + \alpha/2,\alpha} := |u|_0 + |u_x|_0 + |u_t|_0 + \sum_{i,j=1}^n |u_{x_ix_j}|_0 + [u]_{1 + \alpha/2,2+\alpha}< \infty
\end{align}
\end{Def}
 
First we notice that for $\psi: \Sigma \times [0,T) \rightarrow \Hy^{n+1}$ the following flow

\begin{align} \label{IMCF2}
\left ( \frac{\partial \psi}{\partial t} \right)^{\perp} = \frac{\nu}{H}
\end{align}
is, up to tangential diffeomorphisms, equivalent to IMCF  (See Lemma \eqref{equivalence} below ). So the point of this chapter is to prove short time existence to \eqref{IMCF2} which in turn gives us short times existence to \eqref{IMCF} for bounded graphs in Hyperbolic space satisfying bounds mentioned below.

Now if we write $M_t$ as a graph over $\{y = 0\}$ using a function $y:\R^n \times [0,T) \rightarrow \R$ then we have the expressions $\psi(x,t) = (x, y(x,t))$ and $\bar{\nu} = \frac{y(\nabla^0 y,-1)}{\sqrt{1 + |\nabla^0 y|^2}}$ . So we notice that

\begin{align}
\bar{g} \left ( \frac{\partial \psi}{\partial t}, \nu \right )&= \frac{-1}{y\sqrt{1 + |\nabla^0 y|^2}} \frac{\partial y}{\partial t} = \frac{1}{H} \hspace{0.5cm} 
\\\Rightarrow \hspace{0.5cm}\frac{\partial y}{\partial t} &= \frac{- y\sqrt{1 + |\nabla^0 y|^2}}{H}  = \frac{-1}{w H}  = \frac{-vy}{H}
\end{align} 
where we have used the notation $v := \sqrt{1+|\nabla^0 y|^2}$ and the fact that $w = \bar{g}(\partial_y,\bar{\nu}) = \frac{1}{vy}$.

Now if we use the fact that $H = \frac{n + y\tilde{\delta}^{ij} y_{ij}}{v}$, where we denote $\frac{\partial y}{\partial x_i} := y_i$, $\frac{\partial^2 y}{\partial x_i \partial x_j}:=y_{ij}$ and recall that $\tilde{\delta}^{ij} = \delta^{ij} - \frac{y^iy^j}{v^2}$, then we find

\begin{align}\label{IMCF3}
\frac{\partial y}{\partial t} = \frac{-yv^2}{n+y\tilde{\delta}^{ij} y_{ij}} = F(x,y,\nabla^0y, \nabla^0\nabla^0y)
\end{align}
where $F:\R^n\times \R\times \R^n\times \R^{n \times n} \rightarrow \R$, denoted $F(x, u, p_i, a_{ij})$, is a fully nonlinear operator and hence \eqref{IMCF3} is a fully nonlinear parabolic PDE.  

\begin{align} 
\frac{\partial F}{\partial a_{kl}} = \frac{yv^2}{(n+y\tilde{\delta}^{ij} y_{ij})^2} y\tilde{\delta}^{kl}  = \frac{y^2}{H^2}\tilde{\delta}^{kl} 
\end{align}

So if our initial condition $y_0(x) \in \Lambda$ where 
\begin{align}
\Lambda := \{y \in C^2(\R^n):& 0 < H_1 < H(x) < H_2 < \infty, 
\\&0 < y_0 < y(x) \le y_1 < \infty \text{ and } v(x) < v_0 < \infty\}
\end{align}
and $H(x)$ is the mean curvature of the graph of $y(x)$ then we have that $\frac{\partial F}{\partial a_{kl}}  \ge \frac{y_0^2}{H_0^2} \delta_{kl}$ as symmetric matrices and so the linearized operator is uniformly parabolic for functions belonging to $\Lambda$.

Now we state and prove short time existence for \eqref{IMCF3} where we will use the notation that $U_T = \R^n \times [0,T)$ throughout.

\begin{Thm} \label{STE}
Let $F$ be the operator defined above and let $y_0 \in \Lambda \cap C^{2+\alpha}(\R^n)$ where $\alpha \in (0,1)$. Then, for any $0 < \beta < \alpha$, the initial value problem
\begin{align} \label{IMCF4}
\begin{cases}
y_t - F(x,y,\nabla^0y, \nabla^0\nabla^0y) &=0
\\ y(x,0) &= y_0(x)
\end{cases}
\end{align}
has a unique solution $y \in C^{ \frac{2 + \beta}{2}, 2 + \beta}(U_{\epsilon})$, where $\epsilon$ depends only on $\beta$ and $y_0$.
\end{Thm}
\begin{proof}
This proof will be given in three steps. 

\textbf{Step 1:} Let $\hat{y}$ be a solution to the linear parabolic problem
\begin{align}
\begin{cases}
\hat{y}_t  - \Delta \hat{y} &= F(x, y_0,\nabla^0y_0,\nabla^0\nabla^0y_0) - \Delta y_0
\\ \hat{y}(x,0) &= y_0(x)
\end{cases}
\end{align}
by standard linear PDE theory \cite{K} we know that this PDE has a solution $\hat{y} \in C^{2 + \alpha, \frac{2 + \alpha}{2}}(U_T)$ (for any $T > 0$)  with the following bound (independent of $T$)
\begin{align}
\|\hat{y}\|_{\frac{2 +\alpha}{2},2 + \alpha} \le  N(n,\alpha) \left (\|y_0\|_{2 + \alpha} + \|F(y_0)\|_{\alpha} + \|\Delta y_0\|_{\alpha} \right )
\end{align}
where we note that $\|y_0\|_{2 + \alpha} \le C_1, \|\Delta y_0\|_{\alpha} \le C_2$ is implied by our assumptions on $y_0$. 

The bound on $\|F(y_0)\|_{\alpha}$ follows from the fact that $\hat{y}, \nabla^0\hat{y}, \nabla^0\nabla^0 \hat{y} \in C^{\alpha}(\R^{n})$ combined with the fact that if  $u,v \in C^{\alpha}(\R^{n})$ then $uv \in C^{\alpha}(\R^{n})$ and $\frac{u}{v} \in C^{\alpha}(\R^{n})$ as long as $v > v_0 > 0$, is bounded away from zero (Also $F(y_0) = -\frac{v(y_0) y_0}{H(y_0)}$).

Now we can choose $T_0 \le T$ small enough so that for all $t \in [0,T_0]$ 
\begin{align}\label{hatyB}
\hat{y}(\cdot, t) \in \Lambda
\end{align}
where this follows from the fact that $\|\hat{y}\|_{\frac{2 +\alpha}{2},2 + \alpha} \le C$ and hence cannot immediately escape $\Lambda$ by continuity in $t$ of the $C^{2 + \alpha}$ norm.

The idea is that we are going to linearize the nonlinear operator \eqref{IMCF3} at the solution $\hat{y}(\cdot,t)$ and so \eqref{hatyB} implies that $F$ is parabolic at $\hat{y}$.

Now it will also be useful to define $\hat{f}(x,t) \in C^{\alpha, \frac{\alpha}{2}}(U_{T_0})$ to be
\begin{align}
\hat{f} = \hat{y}_t - F(x, \hat{y},\nabla^0\hat{y},\nabla^0\nabla^0\hat{y})
\end{align}
from which we see that $\hat{f}(x,0) = 0$.

\textbf{Step 2:} In this step we would like to employ the Inverse Function Theorem to the map $\Phi : V: = \Lambda \cap C^{2 + \beta, \frac{2 + \beta}{2}}(U_{T_0}) \rightarrow W \subset C^{\beta, \frac{\beta}{2}}(U_{T_0}) \times C^{2 + \beta}(\R^n)$ defined by
\begin{align}
\Phi(y) = \left(y_t - F(x, y,\nabla^0y,\nabla^0\nabla^0y), y(x, 0)\right )
\end{align}
where $V$ is a neighborhood of $\hat{y}$ and $W$ is a neighborhood of $\Phi(\hat{y}) = (\hat{f}, y_0)$.

We notice that $\Phi$ is continuously differentiable on $V$ and its derivative, $D\Phi$ evaluated at $\hat{y} \in V$, is equal to the following operator
\begin{align}
D\Phi(\hat{y}) &: C^{2 + \beta,\frac{2 + \beta}{2}} \rightarrow C^{\beta,\frac{\beta}{2}} \times C^{2 + \beta}
\\D\Phi(\hat{y})[\eta] &= \left (\eta_t- \frac{\partial F}{\partial a_{ij}} \eta_{ij} - \frac{\partial F}{\partial p_i} \eta_i- \frac{\partial F}{\partial u} \eta, \eta(0) \right )
\end{align}
defined for $\eta \in C^{2 + \beta, \frac{2 + \beta}{2}}(U_{T_0})$. We have already explicitly computed $\frac{\partial F}{\partial a_{ij}}$ above and noticed that it was an elliptic operator but we can also calculate $\frac{\partial F}{\partial p_i}$ and $\frac{\partial F}{\partial u}$, as follows.
\begin{align}
\frac{\partial F}{\partial p_i}|_{\hat{y}} &= \frac{-2y y_i}{n+y\tilde{\delta}^{ij} y_{ij}} +  \frac{y^2v^2}{(n+y\tilde{\delta}^{ij} y_{ij})^2} \left (\frac{2y_iy_ky_j}{v^4} y_{kj} -\frac{2y_j}{v^2}y_{ij} \right ) 
\\ &\Rightarrow \hspace{0.25cm}  \left |\frac{\partial F}{\partial p_i} \right | \le \frac{2y}{H} + \frac{2y^2}{H^2} |\nabla^0\nabla^0y|
\\ \frac{\partial F}{\partial u}|_{\hat{y}}  &= \frac{-v^2}{n+y\tilde{\delta}^{ij} y_{ij}} + \frac{yv^2}{(n+y\tilde{\delta}^{ij} y_{ij})^2}\tilde{\delta}^{ij} y_{ij} = \frac{-nv^2}{(n+y\tilde{\delta}^{ij} y_{ij})^2} = \frac{-n}{H^2}  \hspace{0.25cm}
\\&\Rightarrow \hspace{0.25cm}  \left |\frac{\partial F}{\partial u} \right | \le \frac{n}{H^2}
\end{align}
So we see that these coefficients do not present a problem as long as $\hat{y}(\cdot,t) \in \Lambda$, which we confirmed in Step 1, and so the first component of $D\Phi(\hat{y})[\eta]$ is a linear parabolic operator to which standard existence and uniqueness results for linear PDE applies \cite{K} . 

So by standard linear parabolic theory \cite{K} we have that $D\Phi[\hat{y}]$ is one-to-one and onto. Then the inverse function theorem says that there is some $\rho > 0$ so that $\Phi$ is a $C^1$-diffeomorphism from $B_{\rho}(\hat{y}) \subset V$ onto a neighborhood $Z \subset W$ of $(\hat{f},y_0)$.

\textbf{Step 3:} For this step, our goal is to show that the procedure in Step 2 gives us a solution to \eqref{IMCF3} for a short time. For this we let $\epsilon > 0$ and choose $\eta_{\epsilon} \in C^{\infty}([0,1])$ s.t. $0\le \eta_{\epsilon} \le 1$, $0 \le \frac{\partial \eta_{\epsilon}}{\partial t} \le 2 \epsilon^{-1}$,
\begin{align}
\eta_{\epsilon}(t)=
\begin{cases}
0, & 0 \le t \le \epsilon
\\1, &2\epsilon \le t \le 1
\end{cases}
\end{align}
and define $f_\epsilon = \hat{f}\eta_{\epsilon}$. Then, as shown in \cite{BA},\cite{CG2}, $f_{\epsilon} \in C^{\alpha, \frac{\alpha}{2}}(U_{T_0})$ with uniformly bounded norm (in $\epsilon > 0$). Now for each $B_{n}(0) \times [0,T_0]$, $n \in \N$, we can apply Ascoli's theorem to extract a subsequence so that $f_{\epsilon_{k_n}} \rightarrow \hat{f}$ as $\epsilon_{k_n} \rightarrow 0$ in $C^{\beta, \frac{\beta}{2}}(B_{n}(0) \times [0,T_0])$ for all $0 < \beta < \alpha$. Then by choosing a diagonalizing sequence, $\{f_{\epsilon_k}\} = \{f_{\epsilon_{k_k}}\}$, we find $f_{\epsilon_k} \rightarrow \hat{f}$ as $\epsilon_k \rightarrow 0$ in $C^{\beta, \frac{\beta}{2}}(U_{T_0})$, on compact subsets, for all $0 < \beta < \alpha$.

So for small enough $\epsilon$ we have that the pair $(f_{\epsilon},y_0) \in Z$ and hence by Step 2 there exists a unique solution $y^{\epsilon} \in B_{\rho}(\hat{y})$ of the equation
\begin{align}
\Phi(y^{\epsilon}) = (f_{\epsilon},y_0)
\end{align}
which is equivalent to saying that $y$ solves the initial value problem
\begin{align}
y^{\epsilon}_t- F(x,y^{\epsilon},\nabla^0y^{\epsilon}, \nabla^0\nabla^0y^{\epsilon}) &=f_{\epsilon}
\\ y^{\epsilon}(x,0) &= y_0(x,0)
\end{align}
and from the definition $f_{\epsilon} = \hat{f}\eta_{\epsilon}$ for $0 \le t \le \epsilon$ we have that $y^{\epsilon}$ solves the original nonlinear initial value problem \eqref{IMCF4} in $U_{\epsilon} = \R^n \times [0,\epsilon)$.

Then we also know that $y(\cdot,t) \in \Lambda$ for $t \in [0,\epsilon')$ for $0 < \epsilon' \le \epsilon$ since $y \in B_{\rho}(\hat{y})$ and hence cannot immediately escape $\Lambda$. This concludes the proof of existence in Theorem \eqref{STE}.

The proof of uniqueness follows exactly as it does in the compact case so the reader is directed to \cite{BA} or \cite{CG2} for details.

\end{proof}

\begin{Thm} \label{HigherRegularity}
Let $y_0 \in \Lambda \cap C^{m+2+\alpha}(\R^n)$ where $\alpha \in (0,1)$, $m \ge 1$. Then the solution to the initial value problem \eqref{IMCF4} is of class $y \in C^{ \frac{m+2 + \beta}{2},m+ 2 + \beta}(U_{\epsilon})$
\end{Thm}
\begin{proof}
Notice that the arguments in Theorem 2.5.9 in \cite{CG2} are purely local and hence the arguments automatically apply to the non-compact case. 
\end{proof}

\textbf{Note:} The higher regularity in Theorem \ref{HigherRegularity} will be important to us later when we want to apply the Krylov-Safanov estimates to go from $C^2$ estimates to $C^{2,\alpha}$ estimates in Theorem \ref{ContinuationCriterion} since this theorem requires the solution to be at least $C^4$.

\begin{Lem}\label{equivalence}
\eqref{IMCF2} is, up to tangential diffeomorphisms, equivalent to \eqref{IMCF}
\end{Lem}
\begin{proof}
Given a solution $y(\bar{x},t)$ of \eqref{IMCF2} we let $\varphi(x,t)= \left (\bar{x}(x,t),y(\bar{x}(x,t),t) \right )$ where $\bar{x} : \R^n\times[0,T) \rightarrow \R^n$ and then we can find
\begin{align} \label{Relate}
\frac{\partial \varphi}{\partial t} &= \left ( \frac{\partial \bar{x}}{\partial t}, \frac{\partial y}{\partial t}  + \left \langle \nabla^0 y, \frac{\partial \bar{x}}{\partial t} \right \rangle \right ) = \frac{\nu}{H} = \frac{y}{vH} ( \nabla^0y,-1)
\end{align}

This implies that the ODE for $\bar{x}$ is given by
\begin{align} \label{ODE}
\begin{cases}
\frac{\partial \bar{x}}{\partial t}(x,t) &= \frac{y}{vH} \nabla^0 y
\\ \bar{x}(x,0) &=x 
\end{cases}
\end{align}
where we note that this is an ODE since we have already solved \ref{IMCF2} and hence $\frac{y}{vH} \nabla^0 y$ is a predefined, well controlled function. We can confirm this by substituting the second equation given by \ref{Relate} which shows us the following
\begin{align}
\frac{\partial y}{\partial t}  + \left \langle \nabla^0 y, \frac{\partial \bar{x}}{\partial t} \right \rangle = - \frac{y}{vH} \hspace{0.25cm} &\Rightarrow \hspace{0.25cm} \frac{\partial y}{\partial t}  + \frac{y}{vH} |\nabla^0y|^2 = - \frac{y}{vH} 
\\ &\Rightarrow \hspace{0.25cm} \frac{\partial y}{\partial t}  = \frac{-y}{vH} \left ( 1 + |\nabla^0 y|^2 \right ) = \frac{-yv}{H}
\end{align}

So if we define $\textbf{G}(t,\bar{x}) = \frac{y}{vH} \nabla^0 y$ then Theorem \ref{STE} implies that this function is continuous for a short time and hence we can find short time existence  to \eqref{ODE} by standard ODE Theorems. Combining Theorem \ref{STE} with the standard short time existence result for \eqref{ODE} we then obtain short time existence to \eqref{IMCF}, as desired.
\end{proof}

\subsection{$C^0$ and $C^1$ Estimates of IMCF}\label{Subsec:C0C1Estimates}

We start our study of long time existence of non-compact solutions to IMCF in $\Hy^{n+1}$ by looking at a concrete example of the evolution of horospheres in $\Hy^{n+1}$ and then we show that horospheres act as barriers  in $\Hy^{n+1}$ for hypersurfaces satisfying the hypotheses of Theorem 1.

\begin{Ex} \label{HorosphereEx} Consider the horosphere $y = y_0$ as a graph over $\R^n \times \{0\}$. Then $y$ is just a function of time and $H = n$ and so we find the ODE
\begin{align}
\frac{d y}{d t} = \frac{-y}{n}
\end{align}
which has the solution $y(t) = y_0e^{-t/n}$.
\end{Ex}

In order to show that the above family of examples acts as a barrier for bounded graphs as in Theorem \eqref{LTE} we will need to apply the following version of the Omori-Yau maximum principle which will also allow us to apply Theorem \eqref{ODEmax} in order to obtain important estimates throughout this section.

\begin{Thm} \label{OYMax} \cite{P}
Let $(M,g)$ be a complete, non-compact, Riemannian manifold. If $p \in M$ then define $r(x): M \rightarrow \R$ to be the distance from $x$ to $p$ and assume that the radial Ricci curvature satisfies the following bound
\begin{align}\label{OYRicciBound}
Rc(\nabla r,\nabla r) \ge -C(r^2+1)
\end{align}
for some $C > 0$. Then for every bounded above function $u \in C^2(M)$ there is a sequence of points $\{x_n\} \subset M$ so that

\begin{align}\label{OYLaplacianEst}
u(x_n) > \sup_M u - \frac{1}{n} \hspace{0.5cm} |\nabla u|(x_n) < \frac{1}{n} \hspace{0.5cm} \Delta u(x_n) < \frac{1}{n} 
\end{align}

If instead we assume that the sectional curvature of $2$-planes containing $\nabla r$, $K_r$, satisfies the following bound
\begin{align}\label{OYSectBound}
K_r \ge -C(r^2+1)
\end{align}
for some $C > 0$. Then for every bounded above function $u \in C^2(M)$ there is a sequence of points $\{x_n\} \subset M$ so that
\begin{align}\label{OYHessianEst}
u(x_n) > \sup_M u - \frac{1}{n} \hspace{0.5cm} |\nabla u|(x_n) < \frac{1}{n} \hspace{0.5cm} \nabla\nabla u(x_n) < \frac{1}{n} \langle \cdot,\cdot\rangle
\end{align}
\end{Thm}

In the rest of this section we will apply the ODE maximum principle at infinity many times which rests on the application of the Omori-Yau maximum principle to $\Sigma_t$ which we will justify now. Under the assumptions of Theorem \ref{LTE} short time existence, Theorem \ref{STE}, tells us that the bounds in Theorem \ref{LTE} will hold for at least a short time with a maximal existence time of $T < \infty$. 

Then for $t \in [0,T-\epsilon]$, $\epsilon >0$ Theorem \ref{STE} implies $H$ is bounded above and below and $|A|$ is bounded above so the Gauss equations imply that $|Rc|$ and $K_r$ are bounded from below for $t \in [0,T-\epsilon]$ which implies that Theorem \ref{OYMax} applies to $\Sigma_t$ for $t \in [0,T-\epsilon]$. The following estimates will show that it applies for all time.

The following Theorem demonstrates that Examples \ref{HorosphereEx} acts as a barrier for a certain class of non-compact solutions of IMCF. 

\begin{Thm} \label{barriers}
If $\displaystyle 0< \inf_{\R^n} y(x,0) =y_0$ and $\displaystyle \sup_{\R^n} y(x,0) = y_1$  and we assume that $\Sigma_0$ is a hypersurface to which the hypotheses of Theorem 1 apply  then we find that
\begin{align}
y_0 e^{-t/n} \le y(x,t) \le y_1 e^{-t/n}
\end{align}
So horospheres act as barriers for bounded graphs over $\R^n$.
\end{Thm}

\begin{proof}
Notice that by assumption the function $y(x,t)$ is bounded above and below and hence we can use the result of Lemma \ref{ader} that $y_{inf}(t) = \inf_{\R^n} y(x,t)$ is a well defined, locally Lipschitz function. Then by Theorem \ref{OYMax} there exists $\{x_k\} \in \R^n$ a sequence so that $\displaystyle \lim_{k \rightarrow \infty} y(x_k,t) = \inf_{\R^n} y(x,t)$ then we know by the maximum principle at infinity that
\begin{align}
|\nabla^0 y(x_k,t)| < \frac{1}{k} \hspace{1 cm} \nabla^0\nabla^0 y(x_k,t) > -\frac{1}{k} \delta
\end{align}
and so if we use the expressions for $H$ and $w$ in terms of graphs (See \cite{BA}, section 2) we find
\begin{align}
H = \frac{n+y\tilde{\delta}^{ij} y_{ij}}{\sqrt{1 + |\nabla^0y|^2}}  \hspace{0.2 cm} &\Rightarrow  \hspace{0.2 cm} H(x_k,t) \ge\frac{n-k^{-1}y\tilde{\delta}^{ij} \delta_{ij}}{\sqrt{1 + \frac{1}{k^2}}}   \hspace{0.2 cm} \Rightarrow  \hspace{0.2 cm}\lim_{k \rightarrow \infty} H(x_k,t) \ge n
\\ w = \frac{1}{y\sqrt{1 + |\nabla^0y|^2}}  \hspace{0.25 cm} &\Rightarrow  \hspace{0.25 cm} w(x_k,t) = \frac{1}{y(x_k,t) \sqrt{1 + |\nabla^0 y(x_k,t)|^2}}  \hspace{0.25 cm} 
\\&\Rightarrow  \hspace{0.25 cm} \lim_{k \rightarrow \infty}w(x_k,t) = \frac{1}{y_{inf}(t)}
\end{align}

Now  we can find the following ODE for $y(x,t)$
\begin{align}
\frac{\partial}{\partial t} \left ( \frac{1}{y^2} \right ) = \frac{\partial}{\partial t} \bar{g}(\partial_y,\partial_y) = \frac{2}{H} \bar{g}(\bar{\nabla}_{\bar{\nu}} \partial_y,\partial_y) = \frac{2}{H} \bar{g}( - \frac{\bar{\nu}}{y},\partial_y)= \frac{2}{yH} \bar{g}(\bar{\nu},\eta) = \frac{2w}{yH}
\end{align}
where we used the fact that $\bar{\nabla}_X \partial_y = \frac{-1}{y}X$ for any vector field $X$ (See \cite{BA}, section 2, equation 2.3).

If we let $t$ be a point of differentiablility of the locally Lipschitz function $y_{inf}(t)$ and $\{x_k\}$ a sequence such that $y(x_k,t) \rightarrow y_{inf}(t)$ we find that
\begin{align}
\frac{dy_{inf}(t)}{dt}=\lim_{k \rightarrow \infty} \frac{\partial y}{\partial t}(x_k,t) = -\lim_{k \rightarrow \infty} \frac{y^2w}{H}\ge -\frac{1}{n} y_{inf}(t)
\end{align}
and so by using an integrating factor we find
\begin{align}
y_{inf}(t) \ge y_0 e^{-t/n}
\end{align}

Using a similar argument for $y_{sup}(t) = \sup_{\R^n}y(x,t)$ we find the other important estimate.
\end{proof}

\textbf{Note:} Theorem \eqref{barriers} is a simple example where the evolution of $y$ cannot be controlled everywhere but can be controlled at the sup or inf by exploiting the maximum principle at infinity.

Before we move on to gain higher order bounds on the solution $y$ we state all the evolution equations that we require under IMCF, proofs of which can be found in \cite{BA}, section 2.
\begin{Lem}\label{EvolutionEquations}Let $u =\frac{1}{wH}$ and $M_i^j = H A_i^j$ then we can find the following evolution equations under IMCF:
\begin{align}
\frac{\partial g_{ij}}{\partial t} &=2\frac{A_{ij}}{H}
\\ \frac{\partial \nu}{\partial t} &= \frac{\nabla H}{H^2}
\\(\partial_t - \frac{1}{H^2}\Delta) w &= \frac{|A|^2}{H^2} w 
\\(\partial_t - \frac{1}{H^2}\Delta)w^{-1} &= -\frac{|A|^2}{H^2}w^{-1}- \frac{2}{w^{-1}H^2}|\nabla w^{-1}|^2
\\ (\partial_t - \frac{1}{H^2}\Delta) H &=-2 \frac{|\nabla H|^2}{H^3} -\frac{|A|^2}{H} + \frac{n}{H} 
\\  (\partial_t - \frac{1}{H^2}\Delta) u &=2 \frac{g( \nabla w, \nabla u )}{H^2 w^2} - \frac{nu}{H^2}
\\(\partial_t - \frac{1}{H^2}\Delta) A_{ij} &=-\frac{2}{H^3} \nabla_iH \nabla_j H + \left ( \frac{|A|^2}{H^2} + \frac{n}{H^2}\right )A_{ij} 
\\ (\partial_t - \frac{1}{H^2}\Delta) A_i^j &=-\frac{2}{H^3} \nabla_iH \nabla^j H + \left ( \frac{|A|^2}{H^2} + \frac{n}{H^2}\right )A_i^j - \frac{2}{H}A_{il} A^{jl} 
\\ (\partial_t - \frac{1}{H^2}\Delta) M_i^j &=-\frac{2}{H^3} \nabla_iH \nabla^j H - \frac{2g(\nabla H, \nabla ( M_i^j))}{H^3}  +  \frac{2nM_i^j}{H^2} - \frac{2}{H^2}M_{il} M^{jl} 
\\ (\partial_t - \frac{1}{H^2}\Delta)|A|^2 &= -2\frac{|\nabla A|^2}{H^2} -\frac{4}{H^3} A(\nabla H,\nabla H) +2\frac{|A|^4}{H^2} -4\frac{A^3}{H} +2n\frac{|A|^2}{H^2} 
\end{align}
\end{Lem}

Now we obtain $C^1$ bounds on $y$ through the support function $w=\bar{g}(\nu,\eta)$ since $w^{-1} = yv$ (we already have a $C^0$ bound from Theorem \eqref{barriers}).

\begin{Thm}\label{globalwcontrol}
If we assume that $\Sigma_0$ is a hypersurface to which the hypotheses of Theorem \eqref{LTE} apply then we find that
\begin{align}
(i)\hspace{.1 cm} w(x,t) &\ge w_{inf}(0) e^{t/n}
\hspace{2.5 cm}   (ii)\hspace{.1 cm} v(x,t) \le \frac{y_{sup}(0)}{y_{inf}(0)} v_{sup}(0)  
\end{align}
\end{Thm}
\begin{proof}
From the evolution equation for $w^{-1}$ we find
\begin{align}
(\partial_t - \frac{1}{H^2}\Delta)w^{-1} \le -\frac{1}{n}w^{-1}
\end{align}
where we have used that $|A|^2 \ge H^2/n$. Now we can deduce the following differential inequality (at points of differentiability of $w_{sup}$ using Theorem \eqref{ODEmax})
\begin{align}
\frac{d w^{-1}_{sup}}{dt} \le -\frac{1}{n}w^{-1}_{sup}
\end{align}
from which the first estimate follows. Then if we notice that $w^{-1} = vy$ we can find the second estimate by combining with the estimate for $y$ given in Theorem \eqref{barriers}.
\end{proof}

Now we get the required bounds on $H$ which shows that the operator defining IMCF remains uniformly parabolic for $T < \infty$.

\begin{Thm}\label{globalHcontrol} 
If we assume that $\Sigma_0$ is a hypersurface to which the hypotheses of Theorem \eqref{LTE} apply then we find
\begin{align}
c_0 \sqrt{ n^2+ C_0 e^{-2t/n} }  \le H(x,t) \le \sqrt{C_0e^{-2t/n} + n^2}
\end{align}
where $C_0 = H_{sup}(0)^2 - n^2$ if $H_{sup}(0) > n$ and $c_0 =  \frac{y_{inf}(0) H_{inf}(0) }{y_{sup}(0) v_{sup}(0)H_{sup}(0)}$ or
\begin{align}
c_0  \le H(x,t) \le n
\end{align}
where $H_{sup}(0)\le n$ and $c_0 =  \frac{y_{inf}(0) H_{inf}(0) }{y_{sup}(0) v_{sup}(0)}$. 
\end{Thm}
\begin{proof}
We have the evolution equation for $H$ from Lemma \ref{EvolutionEquations}
\begin{align}
\\ (\partial_t - \frac{1}{H^2}\Delta) H &=-2 \frac{|\nabla H|^2}{H^3} -\frac{|A|^2}{H} + \frac{n}{H} 
\end{align}
and by short time existence we know that $H$ is bounded above for at least a short time $t$ and so by using the ODE maximum principle at infinity \eqref{ODEmax} we obtain the differential inequality at points of differentiability of $H_{sup}(t)$
\begin{align}
\frac{d H_{sup}}{dt} \le \frac{1}{nH_{sup}} \left (n^2 - H_{sup}^2 \right )
\end{align}
from which it follows by integration that $H_{sup}(t) \le \sqrt{C_0 e^{-2t/n} + n^2}$ where $C_0 = H_{sup}(0)^2 - n^2$ if $H_{sup}(0) > n$ and $C_0 = 0$  if $H_{sup}(0) \le n$.

Now to obtain the lower bound on $H$ we consider the evolution equation for $u=\frac{1}{wH}$ given in \cite{BA,HI3} and by using the ODE maximum principle at infinity we obtain the following differential inequality at points of differentiability of $u_{sup}$ 
\begin{align} 
\frac{d u_{sup}}{dt} = -\frac{nu_{sup}}{H^2} \le -\frac{n}{n^2 + C_0 e^{-2t/n}}u_{sup}
\end{align}
which implies, by integrating, that $u(x,t) \le \frac{H_{sup}(0) u_{sup}(0)}{\sqrt{n^2 e^{2t/n} + C_0}}$ when $H_{sup}(0) > n$ and then by using the definition of $u = \frac{1}{Hw}$ and applying \eqref{barriers} we find 
\begin{align}
H &\ge \frac{w^{-1}\sqrt{n^2 e^{2t/n} + C_0}}{H_{sup}(0)u_{sup}(0)} = \frac{yv \sqrt{n^2 e^{2t/n} + C_0}}{H_{sup}(0)  u_{sup}(0)} 
\\&\ge \frac{y_{inf}(0) e^{-t/n}H_{inf}(0) w_{inf}(0) \sqrt{n^2 e^{2t/n} + C_0}}{H_{sup}(0) } 
\\&= \frac{y_{inf}(0) H_{inf}(0) }{y_{sup}(0) v_{sup}(0)H_{sup}(0)} \sqrt{n^2  + C_0e^{-2t/n}}
\end{align}
which completes the lower estimate of $H$ when $H_{sup}(0) > n$ . 

When $H_{sup}(0) \le n$ we get the simpler differential inequality at points of differentiability of $u_{sup}$ 
\begin{align} 
\frac{d u_{sup}}{dt} = -\frac{nu_{sup}}{H^2} \le -\frac{u_{sup}}{n}
\end{align}
which implies, by integrating, that $u(x,t) \le u_{sup}(0) e^{-t/n}$ and then  by using the definition of $u = \frac{1}{Hw}$ and applying \eqref{barriers} we find 
\begin{align}
H &\ge w^{-1} u_{sup}(0) e^{-t/n} = yvu_{sup}(0)^{-1} e^{t/n} \ge \frac{y_{inf}(0) H_{inf}(0)}{y_{sup}(0) v_{sup}(0)}
\end{align}
which completes the lower estimate of $H$ when $H_{sup}(0) \le n$.
\end{proof}

\subsection{$C^2$ Estimates and Long Time Existence}\label{subsec:C2LTE}

To obtain an upper bound on $|A|$, the last estimate that we will show, we note that we cannot directly apply the ODE maximum principle at infinity to the maximum eigenvalue of $A$, $\displaystyle \lambda_{max}(x,t)=\max_{v \in T_x\Sigma_t, |v|=1} A(v,v)$ since this function is only locally Lipschitz and hence a laplacian does not exist, even almost everywhere. Since the proof of the ODE maximum principle at infinity relies on a comparison principle for the laplacian we will need to use another method. 

We start by introducing some notation and proving a Proposition which will allow us to construct cutoff functions for IMCF. These cutoff functions will be the key to controling the second fundamental form on a noncompact hypersurface and will allow us to gain $C^2$ control.

We consider the Riemannian manifold $N^{n+1}$ parameterized over $\R_a^b := \{(x_1,...,x_n,y) \in \R^{n+1}: a < y < b\}$ where $a,b \in [-\infty,\infty]$ with  the metric $\bar{g} =\lambda(y)^2 \delta$ which is defined where $\lambda: (a,b) \rightarrow \R$ is defined. We will consider a $n$ dimensional, non-compact hypersurface $\Sigma_0 \subset \R_a^b$.

In line with our previous notation conventions, we will use bars to denote geometric quantities w.r.t $N^{n+1}$, superscript $0$ to denote quantities w.r.t. $\delta$ and no bar or subscript to denote quantities w.r.t. $\Sigma_0$, endowed with the metric induced from $\bar{g}$.

By using well known formulas for conformal metrics, derived from Levi-Civita's formula for the connection, we can find the following expression
\begin{align}\label{exteq}
\bar{\nabla}_X Y = \nabla^0 _XY +\frac{\lambda'}{\lambda}\left (  \langle X,\partial_y \rangle_0 Y + \langle \partial_y,Y \rangle _0 X -  \langle X,Y \rangle_0 \partial_y\right ).
\end{align}

Using this, and the convention that we will put a bar over a vector field $\bar{Z} =\lambda^{-1}(y) Z$ so that $\bar{Z}$ is a unit vector w.r.t. $\bar{g}$, we can obtain the following,
\begin{align}
\bar{div} X &= \bar{g}(\bar{\nabla}_{\bar{e}_i} X, \bar{e}_i)
\\&= \langle \nabla^0 _{e_i} X +\frac{\lambda'}{\lambda} ( \langle e_i,\partial_y \rangle_0 X  + \langle \partial_y, X \rangle_0 e_i -\langle e_i,X\rangle _0 \partial_y, e_i\rangle_0
\\&= div^0 X + (n+1) \frac{\lambda'}{\lambda} \langle X, \partial_y \rangle_0,
\end{align}
where $\{e_1,...,e_{n+1}\}$ is a orthonormal basis for $\R^{n+1}$ w.r.t the flat metric.

We now state and prove a proposition which allows us to define cutoff functions as functions on $\R^{n+1}$ and then compute their evolution equation under IMCF which will be used to define cutoff functions for the flow.

\begin{Prop} \label{EDE}
Let $f: U \rightarrow \R$ where $U \subset \R^n \times \R_+$ open, then the function $g: \Sigma \times [0,T) \rightarrow \R$ defined by $g(p,t) = f(\varphi(p,t))$ has the following evolution equation under IMCF
\begin{align}
&(\partial_t - \frac{1}{H^2}\Delta^{\Sigma_t}) g = \frac{2}{\lambda H}  \nabla^0_{\nu} f+
\\&\frac{1}{\lambda^2 H^2} \left ( \langle \nabla^0_{\nu}\nabla^0 f, \nu \rangle_0-\Delta^0 f -(n-2)\frac{\lambda'}{\lambda} \langle \nabla^0f, \partial_y\rangle_0 -2\frac{\lambda'}{\lambda} \langle \nabla^0 f, \nu \rangle_0 \langle \nu, \partial _y \rangle_0 \right )
\end{align}
\end{Prop}

\begin{proof} 
For any function $u$ and vector field $X$ we have that
\begin{align}
\bar{div}(uX) &= u \text{ }\bar{div}(X) + X(u)
\\ \bar{\nabla}u & = \lambda^{-2}\nabla^0u
\end{align}

Then we notice that 
\begin{align}
\partial_t g &= \frac{1}{H}\bar{\nabla}_{\bar{\nu}} f
\\ \Delta^{\Sigma_t} g &= div(\nabla g) = div(\bar{\nabla}f - \bar{\nabla}_{\bar{\nu}}f \bar{\nu}) = div(\bar{\nabla}f) - H\bar{\nabla}_{\bar{\nu}}f
\end{align}
where $\Delta^{\Sigma_t} = g^{ij} \nabla^{\Sigma_t}\nabla^{\Sigma_t}$, the Laplacian w.r.t. the hypersurface $\Sigma_t$.

\textbf{Note:} Here is where we see a big difference between MCF and IMCF. When studying MCF there is a cancellation between the time derivative term and the first order term in the Laplacian which simplifies computations. In IMCF these two terms combine and hence give an extra term to deal with.

Now we can find the following expression where our goal is to first write all derivatives as extrinsic derivatives in $N^{n+1}$ and then convert all of those derivatives to derivatives on $\R^{n+1}$, using the formulas obtained above. We start by expressing all of the expressions in terms of the covariant derivative of $N^{n+1}$.

\begin{align}
&(\partial_t - \frac{1}{H^2}\Delta^{\Sigma_t}) g =\frac{2}{H}\bar{\nabla}_{\bar{\nu}} f- \frac{1}{H^2}div(\bar{\nabla}f) 
\\&=  \frac{1}{H^2}\left ( -\bar{div}(\bar{\nabla}f) + \langle \bar{\nabla}_{\nu}\bar{\nabla}f, \nu \rangle_0 \right ) + \frac{2}{\lambda H}  \nabla^0_{\nu} f
\end{align}
Now we express all of the terms with respect to the derivative of $\R^{n+1}$.
\begin{align}
&(\partial_t - \frac{1}{H^2}\Delta^{\Sigma_t}) g =
\\&= \frac{1}{H^2}\left ( -  \frac{\bar{div}(\nabla^0 f)}{\lambda^2} - \nabla^0 f(\lambda^{-2}) + \frac{\langle \bar{\nabla}_{\nu}\nabla^0 f, \nu \rangle_0}{\lambda^2} +  \nu(\lambda^{-2}) \langle \nabla^0 f, \nu \rangle_0 \right)
\\&+ \frac{2}{\lambda H}  \nabla^0_{\nu} f
\\&= \frac{1}{H^2}\left ( -\frac{\Delta^0 f}{\lambda^2} - (n+1)\frac{\lambda'}{\lambda^3} \langle \nabla^0 f, \partial_y \rangle  -\langle \nabla^0f,\nabla^0(\lambda^{-2}) \rangle_0 + \langle \nu, \nabla^0(\lambda^{-2})\rangle _0 \langle \nabla^0 f, \nu \rangle_0 \right )
\\&+ \frac{1}{\lambda^2 H^2}\left ( \langle \nabla^0 _{\nu}\nabla^0 f +\frac{\lambda'}{\lambda} \left (  \langle \nu,\partial_y \rangle_0 \nabla^0 f + \langle \partial_y,\nabla^0 f\rangle_0 \nu - \langle \nu,\nabla^0 f \rangle_0 \partial_y \right ) , \nu \rangle_0 \right )\\&+ \frac{2}{\lambda H} \nabla^0_{\nu} f
\\&= \frac{1}{\lambda^2 H^2} \left ( \langle \nabla^0_{\nu}\nabla^0 f, \nu \rangle_0-\Delta^0 f -(n-2)\frac{\lambda'}{\lambda} \langle \nabla^0f, \partial_y\rangle_0 -2\frac{\lambda'}{\lambda} \langle \nabla^0 f, \nu \rangle_0 \langle \nu, \partial _y \rangle_0 \right )\\&+ \frac{2}{\lambda H}  \nabla^0_{\nu} f
\end{align}
\end{proof}

\textbf{Note:} $g$ depends on $t$ through the embedding function $\varphi_t$ but if it also independently depends on $t$ then there will be another term in the evolution equation for $g$ corresponding to the partial derivative w.r.t this aforementioned dependence on $t$.

\textbf{Note:} From now on we will be sloppy and just denote $g$, the function defined on $\Sigma_t$, and $f$, the extrinsically defined function on $N$, as the same function where the composition with the embedding function, $\varphi$, is implied.

Now we make the following definition which we will use throughout the rest of the document.

\begin{Def}\label{impDef} Let $\Sigma_0$ be a hypersurface satisfying the conditions of Theorem \ref{LTE} and let $\Sigma_t$ be the corresponding solution of IMCF which is guaranteed to exist for all time $t \in [0,\infty)$. Then for $T< \infty$ we let 
\begin{align}
\Omega_{R,T} :=B_R(0) \times [0,T)
\end{align}
and then we also define
\begin{align}
\Hc_0 = \displaystyle \inf_{\overline{\Omega}_{\infty,T}} \min (H,H^2) > 0.
\end{align}
If we consider a function $\alpha( x_1,...,x_n,y,t)$ depending on $R,\Hc_0$ then we can also define,
\begin{align}
U_R = \{(x,t) \in \Omega_{R,T} : \alpha(\varphi(x,t),t) > 0 \},
\end{align}
as well as,
\begin{align}
U_{R,\theta,t} = \{(x,t) \in U_R : \alpha(\varphi(x,t),t) > (1-\theta)R^2 \},
\end{align}
where $\theta \in (0,1)$.
\end{Def}

\begin{Lem} \label{CutEq}
If we define $\alpha =\frac{1}{R} \left( R^2 - |x|^2 - \frac{2}{\Hc^R_0} (ny_0^2 + 4y_0R + C_R)t\right)$ for $N = \Hy^{n+1}$, where $y(x,0) \le y_0$, $C_R \ge 0$ is arbitrary. Then $\alpha$ is a subsolution to the IMCF heat operator on $\Sigma_t$, i.e. for $t \in [0,T)$:
\begin{align}
\left (\partial_t - \frac{1}{H^2}\Delta \right )\alpha \le -\frac{2C_R}{R\Hc_0} \le 0
\end{align}
\end{Lem}

\begin{proof}
 If we let $|x|^2 = x_1^2 + ... + x_n^2$  for $N = \Hy^{n+1}$, in the upper half space model, then we find
\begin{align}
(\partial_t - \frac{1}{H^2}\Delta) |x|^2&= \frac{1}{H^2} \left ( y^2 ( 2|\hat{\nu}|^2 - 2n ) +4y \langle x,\nu\rangle _0 \langle \nu,\partial y\rangle_0 \right ) + \frac{4y}{H} \langle x,\nu\rangle_0
\end{align}
where we have used the following relations as well as Proposition \eqref{EDE}
\begin{align}
\nabla^0 |x|^2 &=2x \hspace{0.5cm} \nabla^0_{\nu} \nabla^0 |x|^2 = 2 \hat{\nu} \hspace{0.5cm}  \Delta^0 |x|^2 = 2n
\end{align}
where $\hat{\nu}$ is  the projection of $\nu$ onto $\R^n \times\{0\}$.
\begin{align}
&(\partial_t - \frac{1}{H^2}\Delta)\alpha =\frac{-1}{RH^2} \left ( y^2 ( 2|\hat{\nu}|^2 - 2n ) +4y \langle x,\nu\rangle _0 \langle \nu,\partial y\rangle_0 \right ) - \frac{4y}{RH} \langle x,\nu\rangle_0 
\\&- \frac{2}{R\Hc_0}(ny_0^2 + 4y_0R+C_R)
\\&\le \frac{2ny^2}{RH^2} + \frac{4yR}{RH^2}+ \frac{4yR}{RH}- \frac{2}{R\Hc_0}(ny_0^2 + 4y_0R + C_R) \le -\frac{2C_R}{R\Hc_0}
\end{align}
\end{proof}

\textbf{Note:} We purposefully leave $C_R> 0$ undefined for now because we will choose it later depending on the estimate we are trying to achieve. 

Now we will obtain a local second order estimate through bounding $A_{ij}$. Again we will need to consider $P_i^j = w^{-1} A_i^j$ instead of $A_i^j$ directly because we need to leverage the good evolution equation for $w^{-1}$ in order to kill the bad terms in the evolution of $A_i^j$ and obtain a useful evolution equation for $P_i^j$. We start by obtaining important evolution equations and then obtain the estimate in Lemma \ref{LAEs}.

\begin{Lem} \label{LAEq} If we define $P_i^j = w^{-1}A_i^j$ then we will find the following evolution equation
\begin{align}
(\partial_t - \frac{1}{H^2}\Delta)P_i^j&= -\frac{2}{wH^3} \nabla_iH \nabla^jH -\frac{2w}{H^2} g(\nabla w^{-1},\nabla P_i^j ) 
+\frac{n}{H^2}P_i^j - \frac{2w}{H}(P^2)_i^j
\end{align}

Now if we consider $\eta : \R \rightarrow \R$ and $\alpha$ the cutoff function from Lemma \ref{CutEq} then we find the following evolution equation for $\eta(\alpha) P_i^j$
\begin{align}
&(\partial_t - \frac{1}{H^2}\Delta)(\eta P_i^j)=-\frac{2w}{H^2} g(\nabla w^{-1}, \nabla(\eta P_i^j)) -\frac{2}{\eta H^2} g(\nabla \eta, \nabla (\eta P_i^j))
\\& -\frac{2\eta}{wH^3} \nabla_iH \nabla^jH+\frac{2w\eta'}{H^2}P_i^jg(\nabla w^{-1},\nabla\alpha)+\frac{P_i^j}{H^2}\left(\frac{2\eta'^2}{\eta} - \eta''\right) |\nabla \alpha|^2 
\\&+\frac{n}{H^2}(\eta P_i^j) - \frac{2w}{\eta H}(\eta^2P^2)_i^j-\frac{2C_R\eta' P_i^j}{R\Hc_0}
\end{align}
\end{Lem}

\begin{proof}

\begin{align}
&(\partial_t - \frac{1}{H^2}\Delta)P_i^j=w^{-1} (\partial_t - \frac{1}{H^2}\Delta)A_i^j +A_i^j(\partial_t - \frac{1}{H^2}\Delta)w^{-1} -\frac{2}{H^2} g(\nabla w^{-1}, \nabla A_i^j)
\\&=w^{-1} \left ( -\frac{2}{H^3} \nabla_iH \nabla^jH + \frac{|A|^2}{H^2} A_i^j + \frac{n}{H^2} A_i^j -\frac{2}{H} (A^2)_i^j \right )
\\& -\frac{|A|^2}{H^2}w^{-1}A_i^j -\frac{2wA_i^j}{H^2}|\nabla w^{-1}|^2-\frac{2}{H^2} g(\nabla w^{-1}, \nabla A_i^j)
\\&= -\frac{2}{wH^3} \nabla_iH \nabla^jH -\frac{2wA_i^j}{H^2}|\nabla w^{-1}|^2-\frac{2}{H^2} g(\nabla w^{-1}, \nabla A_i^j)
+\frac{n}{H^2}P_i^j - \frac{2w}{H}(P^2)_i^j
\end{align}

Now we will use the fact that
\begin{align}
-\frac{2w}{H^2} g(\nabla w^{-1},\nabla (w^{-1}A_i^j ) ) = -\frac{2wA_i^j}{H^2}|\nabla w^{-1}|^2 -\frac{2}{H^2} g(\nabla w^{-1}, \nabla A_i^j)
\end{align}
to find the following
\begin{align}
(\partial_t - \frac{1}{H^2}\Delta)P_i^j &= -\frac{2}{wH^3} \nabla_iH \nabla^jH -\frac{2w}{H^2} g(\nabla w^{-1},\nabla P_i^j ) 
+\frac{n}{H^2}P_i^j - \frac{2w}{H}(P^2)_i^j
\end{align}

\textbf{Note:} We are not worried about the $\nabla_iH \nabla^jH$ term since at some point in this argument we are going to look at the maximum eigenvalue of $P_i^j$ in which case this term will be negative.

Now if we let $\alpha$ be the cutoff function from \eqref{CutEq} so that $(\partial_t - \frac{1}{H^2}\Delta)\alpha \le -\frac{2C_R}{R\Hc_0}$ then we can compute the following evolution inequality for $\alpha P_i^j$
\begin{align}
(\partial_t - \frac{1}{H^2}\Delta)(\alpha P_i^j)&\le \alpha (\partial_t - \frac{1}{H^2}\Delta)P_i^j +P_i^j(\partial_t - \frac{1}{H^2}\Delta)\alpha -\frac{2}{H^2} g(\nabla \alpha, \nabla P_i^j)
\\&= -\frac{2\alpha}{wH^3} \nabla_iH \nabla^jH -\frac{2w\alpha}{H^2} g(\nabla w^{-1},\nabla P_i^j ) 
-\frac{2}{H^2} g(\nabla \alpha, \nabla P_i^j)
\\&+\frac{n\alpha}{H^2}P_i^j - \frac{2w\alpha}{H}(P^2)_i^j-\frac{2C_RP_i^j}{R\Hc_0}
\end{align}

Now we again compute some gradient terms
\begin{align}
-\frac{2w}{H^2} g(\nabla w^{-1}, \nabla(\alpha P_i^j)) &=-\frac{2\alpha w}{H^2}g(\nabla w^{-1},\nabla P_i^j) -\frac{2w}{H^2}P_i^jg(\nabla w^{-1}\nabla\alpha)
\\ -\frac{2}{\alpha H^2} g(\nabla \alpha, \nabla (\alpha P_i^j)) &= -\frac{2}{H^2} g(\nabla \alpha,\nabla P_i^j) - \frac{2P_i^j}{\alpha H^2} |\nabla \alpha|^2
\end{align}
from which we find
\begin{align}
(\partial_t - \frac{1}{H^2}\Delta)(\alpha P_i^j)&\le -\frac{2w}{H^2} g(\nabla w^{-1}, \nabla(\alpha P_i^j)) -\frac{2}{\alpha H^2} g(\nabla \alpha, \nabla (\alpha P_i^j))
\\& -\frac{2\alpha}{wH^3} \nabla_iH \nabla^jH+\frac{2w}{H^2}P_i^jg(\nabla w^{-1},\nabla\alpha)+\frac{2P_i^j}{\alpha H^2} |\nabla \alpha|^2
\\&+\frac{n}{H^2}(\alpha P_i^j) - \frac{2w}{\alpha H}(\alpha^2 P^2)_i^j-\frac{2C_RP_i^j}{R\Hc_0}
\end{align}

To deal with the $\alpha$ that shows up in the denominator of the $|\nabla \alpha|^2$ term we consider a function $\eta : \R \rightarrow \R$ non-decreasing and compute the following evolution for $\eta(\alpha) P_i^j$
\begin{align}
&(\partial_t - \frac{1}{H^2}\Delta)(\eta P_i^j)=-\frac{2w}{H^2} g(\nabla w^{-1}, \nabla(\eta P_i^j)) -\frac{2}{\eta H^2} g(\nabla \eta, \nabla (\eta P_i^j))
\\& -\frac{2\eta}{wH^3} \nabla_iH \nabla^jH+\frac{2w\eta'}{H^2}P_i^jg(\nabla w^{-1},\nabla\alpha)+\frac{2P_i^j}{H^2}\frac{\eta'^2}{\eta} |\nabla \alpha|^2 - \frac{\eta'' P_i^j}{H^2} |\nabla \alpha |^2
\\&+\frac{n}{H^2}(\eta P_i^j) - \frac{2w}{\eta H}(\eta^2P^2)_i^j-\frac{2C_R\eta' P_i^j}{R\Hc_0^R}
\end{align}

\end{proof}

Now we are ready to prove an estimate for $P_i^j$ which will imply an estimate for $A_i^j$.

\begin{Lem}\label{LAEs}
Define $P_i^j = w^{-1} A_i^j$ and assume that $\Sigma_0$ is a hypersurface to which Theorem \ref{LTE} applies on $B_R$ and let $T$ be the maximal time of existence. Then for $t \in [0,T)$, $\theta \in (0,1)$,
\begin{align}
\max_{U_{R,\theta,t}} P_i^j &\le \max \left ( \max_{U_{R,1,0}} P_i^j, c_0 \right )(1-\theta)^{-2},
\end{align}
where $\displaystyle \max_U P_i^j$ refers to the maximum eigenvalue of $P$ over the set $U$ and $c_0$ is a upper bound on $(Hw)^{-1}$ in $U_R$, guaranteed by previous estimates. 
\end{Lem}
\begin{proof}
Now we would like to better understand some terms in the equation given in Lemma \ref{LAEq} and estimate the bad terms starting with the following term
\begin{align}
\frac{2w\eta'}{H^2}g(\nabla w^{-1},\nabla\alpha) = \frac{-2w\eta'}{H^2}g(\nabla w^{-1},\nabla\alpha)\le \frac{2w\eta'}{H^2}|\nabla w^{-1}||\nabla \alpha|.
\end{align}

Now we can use the fact that $\Sigma_t$ is well defined hypersurface expressed as a graph over a plane with bounded gradient, $v$, in $\Omega_{R,T}$ to deduce that there exists a $D > 0$ so that $D^{-2} \delta \le g \le D^2 \delta$, in $\Omega_{R,T}$. Hence $|\nabla \alpha| \le D|\nabla^0 \alpha| \le 2D\frac{|x|}{R} \le 2D$. We note that $D$ depends on upper and lower bounds on $y$ and a upper bound on $v$ since $g_{ij} = \frac{1}{y^2} \left (\delta_{ij} + y_iy_j \right )$ and note that Lemma \ref{globalwcontrol} and Theorem \ref{barriers} will give use the desired control.

Now we note that $|\nabla w^{-1}| = \frac{|\nabla w|}{w^2} \le D'$ in $\Omega_{R,T}$, which is equivalent to having a lower bound on $w$, which follows from Theorem \ref{globalwcontrol}, and a bound on $|A|^2$, which can be seen by choosing a vector $v$ tangent to $\Sigma_t$ and calculating,
\begin{align}
\nabla_v w = \nabla_v \bar{g}(\nu,\eta) = \bar{g}(\bar{\nabla}_v\nu,\eta^T) = A(v, \eta^T) \hspace{0.5cm} \Rightarrow \hspace{0.5cm}|\nabla w|^2 \le |A|^2.
\end{align}
For this term we use short time existence and the fact that $t < T$ to deduce a bound on $|A|$. 

Now if we choose $\eta(s) = s^2$ then the term $\frac{2 \eta'^2}{\eta} - \eta'' =2$ and so we can find
\begin{align}
&\frac{2 \eta'w}{H^2} |\nabla w^{-1}||\nabla \alpha | + \frac{|\nabla \alpha |^2}{H^2} \left ( \frac{2 \eta'^2}{\eta} - \eta'' \right ) - \frac{2C_R\eta'}{R\Hc_0} 
\\&\le \frac{1}{R\Hc_0} \left ( 2 \eta'w D D'R  - 2C_R\eta'\right )+\frac{8 D^2}{H^2}
\\&\le \frac{2\eta'}{R\Hc_0} \left (  D D' D'' R - C_R\right)+\frac{8 D^2}{\Hc_0}
\end{align}
where $D''$ is a upper bound on $w$ (implied by Lemma \ref{barriers}). Now we can choose $C_R \ge  D D' D'' R $ in order to get rid of the bad gradient terms that come from the cutoff function. The other term $\frac{8 D^2 }{H^2}$ will be dealt with as part of the zero order terms below.

 Now we look to understanding the zero order terms $\frac{n+8 D^2}{H^2}(\eta P_i^j) - \frac{2w}{\eta H} (\eta^2P^2)_i^j$. Now if we let $\lambda$ be the largest eigenvalue of $P_i^j$ at a point $(x,t) \in U_R$, then we find the following
\begin{align}
-\frac{2w}{\eta H}(\eta \lambda)^2+ \frac{n+8 D^2}{H^2} (\eta \lambda) &= -\frac{2w}{ H}\lambda \left (\eta \lambda - \frac{1}{2}(n+8 D^2) \eta (H w)^{-1}\right ) 
\\&\le -\frac{2w}{ H}\lambda \left (\eta \lambda - c_0 \eta \right ) 
\end{align}
at the point $(x,t)$ where $c_0$ is an upper bound on $\frac{1}{2}(n+8 D^2)(Hw)^{-1}$, as in the statement of the theorem. So we notice that this term is negative when $\eta \lambda > c_0 \eta$ and hence decreasing. We will use this intuition about the zero order terms later but we just take note of it now and move on to make this argument rigorous.

Now we are ready to give the proof of the lemma. Let $\Phi_i^j = C\delta_i^j -\eta P_i^j +  \epsilon(t- \tau)\delta_i^j $ where $C = \max \left (C_0, c_0 \eta \right )$ and $C_0$ is the maximum eigenvalue of $\eta P_i^j$ in the set $U_{R,1,0}$ and $\tau \ge 0$ will be chosen later. The goal is to show that the minimum eigenvalue of $C\delta_i^j -\eta P_i^j$ is positive.

For sake of contradiction assume that the minimum eigenvalue over $\overline{U}_R$ of $\Phi_i^j$ is negative. Then we first consider the case where there is a point $(x_0,t_0) \in U_R$ where $\Phi_i^j$ has a zero eigenvector, call it $\beta$, for the first time with eigenvector $v \in T_{x_0}\Sigma_{t_0}$. Then we use parallel translation to extend $v$ along radial geodesics emanating from $x_0 \in \Sigma_t$  in a neighborhood of $x_0$ and then extend it to be constant in time for a short amount of time. From this construction we find the following inequalities,
\begin{align}
\frac{\partial v}{\partial t}|_{(x_0,t_0)} &= 0 \hspace{0.5cm} \nabla v|_{(x_0,t_0)} = 0\hspace{0.5cm} \frac{\partial \Phi(v,v)}{\partial t} |_{(x_0,t_0)} \le 0 \hspace{0.5cm} 
\\&\nabla \Phi(v,v)|_{(x_0,t_0)} = 0\hspace{0.5cm} \Delta \Phi(v,v)|_{(x_0,t_0)} \ge 0.
\end{align}

We can also compute that,
\begin{align}
&\Delta(\Phi(v,v)) = g^{ij} \nabla_i \left ((\nabla_j\Phi)(v,v) + 2 \Phi(\nabla_jv,v) \right )
\\&= g^{ij} \left ( (\nabla_i\nabla_j \Phi)(v,v) + 4(\nabla_j\Phi)(\nabla_iv,v) + 2 \Phi(\nabla_i\nabla_j v,v) + 2 \Phi(\nabla_iv,\nabla_jv)\right )
\\&= (\Delta \Phi)(v,v) + 4(\nabla \Phi)(\nabla v,v) + 2\Phi(\Delta v,v) + 2\Phi(\nabla v, \nabla v),
\end{align}
and hence we find,
\begin{align}
 \Delta(\Phi(v,v))|_{(x_0,t_0)} &\ge2\Phi(\Delta v,v)|_{(x_0,t_0)} = 0,
\end{align}
where we used the fact that $v$ is a zero eigenvector for $\Phi$ at the point $(x_0,t_0)$ in the last equality.

Then we find the following evolution inequality at the point $(x_0,t_0)$,
\begin{align}
(\partial_t - \frac{1}{H^2}\Delta)(\eta \Phi_i^jv^iv_j)&\ge \frac{2w}{\eta H}(\eta^2P^2)_i^jv^iv_j - \frac{n+8D^2}{H^2}(\eta P_i^jv^iv_j) + \epsilon\delta_i^jv^iv_j.
\end{align}
Now notice we find the following inequality at the point $(x_0,t_0)$ where we let $\lambda = P_i^jv^iv_j$,
\begin{align}
\frac{2w}{\eta H}(\eta^2P^2)_i^jv^iv_j - \frac{n+8D^2}{H^2}(\eta P_i^jv^iv_j) +  \epsilon\delta_i^jv^iv_j = \frac{2w}{ H}\lambda \left (\eta \lambda - c_0 \eta \right ) +  \epsilon\delta_i^jv^iv_j > 0,
\end{align}
where the strict inequality follows since $C$ was chosen to be larger than $c_0 \eta$, $\beta = C - \lambda \eta +\epsilon (t_0 - \tau)\delta_i^jv^iv_j = 0$ so $ \lambda \eta = C + \epsilon (t_0 - \tau)\delta_i^jv^iv_j$ and by choosing $\tau$ so that $t_0 - \tau >0$.

By our assumptions though we know that $ \frac{\partial \Phi(v,v)}{\partial t} |_{(x_0,t_0)} \le 0$ and $ \Delta \Phi(v,v)|_{(x_0,t_0)} \ge 0$ and hence we find
\begin{align}
(\partial_t - \frac{1}{H^2}\Delta) \Phi(v,v) \le 0
\end{align}
which is a contradiction so if we let $\epsilon \rightarrow 0$ we see that $C\delta_i^j -\eta P_i^j$ cannot attain a strictly negative eigenvalue on $U_R$.

Now we know that $C\delta_i^j -\eta P_i^j$ cannot obtain a strictly negative eigenvalue on $\{\alpha = 0 \}$ and we see by construction that $C\delta_i^j -\eta P_i^j$ does not obtain a negative eigenvalue at time $t = 0$ since $C$ was chosen to be less than $C_0$, the minimum eigenvalue of $\eta P_i^j$ in the set $U_{R,1,0}$. So it doesn't obtain one anywhere on $U_R$ and hence $\eta P_i^j$ is bounded from above, as desired. 

More specifically we have that,
\begin{align}
\max_{U_{R,1,t}}\eta P_i^j \le C = \max \left ( \max_{U_{R,1,0}} \eta P_i^j, c_0 \eta \right )
\end{align}

Since we know that $0 \le \alpha \le R^2$ on $U_R$ we know that $\eta(\alpha) \le  R^4$ on $U_R$ and since $\alpha \ge (1 - \theta) R^2$ on the set $U_{R,\theta,t}$ we know that $\eta(\alpha) \ge (1 - \theta)^2 R^4$ and so we have
\begin{align}
 \max_{U_{R,\theta,t}}P_i^j \le C = \max \left ( \max_{U_{R,1,0}} P_i^j, c_0 \right )(1-\theta)^{-2},
\end{align}
which yields the desired result.

\end{proof}

\begin{Cor}\label{globalAcontrol}
If we assume that $\Sigma_0$ is a hypersurface to which the hypotheses of Theorem \eqref{LTE} apply then we find
\begin{align}
|A| \le C
\end{align}
where $C$ depends on the initial data but not on $T$.
\end{Cor}
\begin{proof}
By Lemma \ref{LAEs} we find the bound,
\begin{align}
 \max_{U_{R,\theta,t}}P_i^j \le C = \max \left ( \max_{U_{R,1,0}} P_i^j, \frac{nc_0}{2} \right )(1-\theta)^{-2},
\end{align}
and since $\displaystyle \sup_{\Sigma_0} P_i^j \le C$ by assumption we can take the limit as $R \rightarrow \infty$ to find,
\begin{align}
 \max_{U_{\infty,\theta,t}}P_i^j \le C = \max \left ( \sup_{\Sigma_0} P_i^j, c_0 \right )(1-\theta)^{-2},
\end{align}
but one may be concerned that $U_{\infty}$ could degenerate, i.e. $\R^n \times [0,\epsilon) \not \subset U_{\infty}$ for every $\epsilon >0$. To see that this cannot happen 
we can find the following characterization of the largest time $t$ that can occur in $U_R$,
\begin{align}
\alpha > 0 \hspace{0.25cm} &\Rightarrow\hspace{0.25cm} R^2 - |x|^2 - \frac{2}{\Hc_0} \left (ny_0^2 + 4y_0R + C_R \right )t > 0 \\&\Rightarrow\hspace{0.25cm} t < \frac{\Hc_0 (R^2 - |x|^2)}{2(ny_0^2 + 4y_0R + R C )},
\end{align}

which for fixed $x$ has a limit as $R \rightarrow \infty$ and tells us that $t < \infty$ and hence $U_{\infty}$ is non-degenerate.
By unpacking the definition of $P_i^j$ and using the bounds on $w$ from Theorem \ref{barriers} and Lemma \ref{globalwcontrol} we find the desired estimate for $|A|$.
\end{proof}

Now we can prove a long time existence theorem.
\begin{Thm} \label{ContinuationCriterion}
Let $\Sigma_0$ be a hypersurface satisfying the hypotheses of Theorem \eqref{LTE} then $\Sigma_t$, the corresponding solution to IMCF, exists for all time $t \in [0,\infty)$.
\end{Thm}
\begin{proof}
Assume that $T$ is the maximal existence time and using the upper and lower bounds on $H$ of Theorem \ref{globalHcontrol}, combined with  upper bound on $|A|$ of Theorem \ref{globalAcontrol} we find $C^2$ control on the solution $\Sigma_t$. Then since $y \ge c >0$ for $T < \infty$ by Theorem \ref{barriers} we know that \eqref{IMCF3} is uniformly parabolic and so we combine the $C^2$ control with the results of Krylov \cite{K2} to obtain $C^{2,\alpha}$ control on $\Sigma_t$ and hence if we consider a sequence of times $T_k \in [0,T)$ so that $T_k \nearrow T$ then we know that $\Sigma_{T_k} \rightarrow \Sigma_T$ in $C^{2,\alpha}$ where $\Sigma_T$ is a $C^{2,\alpha}$ hypersurface. Then by short time existence applied to $\Sigma_T$ we can extend the flow beyond time $T$, contradicting the assumption that $T$ was the maximal existence time.
\end{proof}

\subsection{Asymptotic Properties}\label{Subsec:Asymptotic}

Now we move on to discussing asymptotic analysis where our goal is to state precise theorems with brief proofs but the reader can refer to \cite{BA,BHW,CG1,CG3,S,U} for further details. We start with a $C^1$ asymptotic estimate.

\begin{Lem}\label{AG}
For hypersurfaces satisfying the hypotheses of Theorem \eqref{LTE} the corresponding solution to IMCF in hyperbolic space satisfies
\begin{align}
v^2-1=|\nabla^0 y|^2 \le Ce^{-2t/n}
\end{align}
\end{Lem}
\begin{proof}
If we define $\psi = y_{x_1}^2 + ...+y_{x_n}^2 = |\nabla^0 y|^2$ and then differentiate the equation $\frac{\partial y}{\partial t} = \frac{-1}{F}$ w.r.t $y^k \nabla^0_k$, where $F = \frac{ny^{-1} + \tilde{\delta}^{ij}y_{ij}}{v^2} = \frac{H}{vy}$ we find
\begin{align}
\frac{\partial \psi}{\partial t}&= y^k \left ( \frac{\partial y}{\partial t} \right ) _k = y^k \left ( \frac{-1}{F} \right ) _k  = \frac{1}{F^2} y^k F_k
\\&= \frac{1}{v^2F^2}y^k \left ( -2F y^ly_{lk} - ny^{-2}y_k +\tilde{\delta}^{ij} y_{ijk} -2\frac{y^i_k y^jy_{ij}}{v^2} +2 \frac{y^iy^jy_{ij}y^ly_{lk}}{v^4}\right )
\\&= \frac{1}{v^2F^2}\left ( \tilde{\delta}^{ij} y_{ijk}y^k  + 2 G^k \psi_k - \frac{2n \psi}{y^2}\right )
\end{align}
where notice that $\psi_k = y^ly_{lk}$ and we have that $G^k = -Fy_k  -\frac{1}{v^2} y_{jk}y^j +\frac{1}{v^4}y^iy^jy_{ij}y_k$.

Now if we also notice the following
\begin{align}
\tilde{\delta}^{ij} \psi_{ij} = \tilde{\delta}^{ij}(y_{kij}y^k + y^k_jy_{ki} ) = \tilde{\delta}^{ij}y_{ijk}y^k +\tilde{\delta}^{ij} y^k_jy_{ki}
\end{align}
where we notice that the difference between this case and the graph over a sphere case is that we don't get an extra term from commuting derivatives in our case.

We can also rewrite $\tilde{\delta}^{ij} y^k_jy_{ki}$ in the following way
\begin{align}
\tilde{\delta}^{ij} y^k_jy_{ki} = \delta^{lm} \tilde{\delta}^{ij} y_{li}y_{mj} = \delta^{lm}\delta^{ij} y_{li}y_{mj} - \frac{\delta^{lm}}{v^2} y^iy_{li}y^jy_{mj} = \delta^{ij}\delta^{lm}y_{il}y_{jm} - \frac{1}{v^2}\psi^k\psi_k
\end{align}

So that we now obtain the desired evolution equation
\begin{align}\label{GradientEquation}
\frac{\partial \psi}{\partial t}&= \frac{1}{v^2F^2}\left ( \tilde{\delta}^{ij} \psi_{ij}  + 2 G^k \psi_k + \frac{1}{v^2}\psi^k\psi_k - \frac{2n \psi}{y^2} -  \delta^{ij}\delta^{lm}y_{il}y_{jm}\right )
\end{align}

Now we can use this and Theorem \eqref{ODEmax} to derive a differential inequality for $\psi_{sup}(t)$, at points of differentiability
\begin{align}
\frac{d \psi_{sup}}{dt} \le \frac{-2n}{H^2} \psi_{sup} \le \frac{-2n}{n^2 + C_0e^{-2t/n}} \psi_{sup} \le -2 \left ( \frac{1}{n} - \bar{C} e^{-2t/n}\right) \psi_{sup}(t)
\end{align}
where we have used the bound $H^2 \le n^2 + C_0e^{-2t/n}$ and chosen a constant $\bar{C} > 0$. 

Now by integrating this differential inequality we find
\begin{align}
\psi_{sup} \le D e^{-2t/n - ne^{-2t/n}}
\end{align}
for some constant $D>0$ which implies that $\psi=|\nabla^0 y|^2= O(e^{-2t/n})$, as desired.
\end{proof}

Now we move from $C^1$ bounds to $C^2$ bounds.

\begin{Cor}\label{CrudeC^2Bound}
For hypersurfaces satisfying the hypotheses of Theorem \eqref{LTE} the corresponding solution to IMCF in hyperbolic space satisfies
\begin{align}
|\nabla^2y| \le C e^{t/n}
\end{align}
\end{Cor}
\begin{proof}
For graphs we can write $A_{ij} = \frac{1}{yv} y_{ij} + g_{ij}$ and $g^{ij} = y^2\tilde{\delta}^{ij}$ and hence $A_i^j = \frac{y}{v}\tilde{\delta}^{ik}y_{kj} + \delta_i^j$. So we can rewrite $\tilde{\delta}^{ik}y_{kj} = \frac{v}{y} ( A_i^j-\delta_i^j$ from which it follows that $|\nabla^2y| \le \frac{v|A|}{y} \le Ce^{t/n}$.
\end{proof}

Now we would like to improve on the $C^2$ bound of Corollary \ref{CrudeC^2Bound}.

\begin{Lem}\label{C^2Decay}
For hypersurfaces satisfying the hypotheses of Theorem \eqref{LTE} the corresponding solution to IMCF in hyperbolic space satisfies
\begin{align}
|\nabla^0\nabla^0y|\le C e^{-\gamma t}
\end{align}
for $(x,t) \in \R^n\times[0,\infty)$, $\gamma > 0$.
\end{Lem}
\begin{proof}
Remember we can rewrite IMCF in terms of $y$ as the following PDE
\begin{align}\label{IMCF5}
\frac{\partial y}{\partial t} = \frac{-yv^2}{n+y\tilde{\delta}^{ij} y_{ij}} = F(y,\nabla^0y, \nabla^0\nabla^0y)
\end{align}
where $\tilde{\delta}^{ij}=\delta^{ij} - \frac{\nabla^0_iy\nabla^0_jy}{v^2}$ and $F(u,p_k,a_{lm})$. Now we aim to derive an equation for $\beta = y^{ij}y_{ij}$ which we find by differentiating \eqref{IMCF5} twice and contracting it with $y^{ij}$ we find
\begin{align}
\beta_t&= 2\frac{\partial F}{\partial a_{lm}} y_{ijlm}y^{ij}+2\frac{\partial F}{\partial p_k}y_{ijk}y^{ij} +2\frac{\partial F}{\partial u}y_{ij}y^{ij}
\\&+2\frac{\partial^2 F}{\partial u^2} y_iy_jy^{ij} + 2\frac{\partial^2 F}{\partial p_k\partial p_n}y_{ik}y_{jn}y^{ij} +2 \frac{\partial^2 F}{\partial a_{lm}\partial a_{mn}}y_{ilm}y_{jnm}y^{ij}
\\&+4\frac{\partial^2 F}{\partial u \partial p_k}y_{ik}y_jy^{ij} + 4\frac{\partial^2 F}{\partial a_{lm}\partial u} y_{ilm}y_j y^{ij} + 4\frac{\partial^2 F}{\partial a_{lm} \partial p_k} y_{ilm}y_{jk}y^{ij}
\end{align}
which can be rewritten if we notice that 
\begin{align}
\beta_l&=2y_{ijl}y^{ij}
\\ \beta_{lm}&= 2(y_{ijl}y^{ijm} + y_{ijlm}y^{ij})
\end{align}
and hence
\begin{align}
\beta_t&= 2\frac{\partial F}{\partial a_{lm}} \beta_{lm}+2\frac{\partial F}{\partial p_k}\beta_k +2\frac{\partial F}{\partial u}\beta-2\frac{\partial F}{\partial a_{lm}}y_{ijl}y^{ijm}
\\&+2\frac{\partial^2 F}{\partial u^2} y_iy_jy^{ij} + 2\frac{\partial^2 F}{\partial p_k\partial p_n}y_{ik}y_{jn}y^{ij} +2 \frac{\partial^2 F}{\partial a_{lm}\partial a_{no}}y_{ilm}y_{jno}y^{ij}
\\&+4\frac{\partial^2 F}{\partial u \partial p_k}y_{ik}y_jy^{ij} + 4\frac{\partial^2 F}{\partial a_{lm}\partial u} y_{ilm}y_j y^{ij} + 4\frac{\partial^2 F}{\partial a_{lm} \partial p_k} y_{ilm}y_{jk}y^{ij}
\end{align}
Now our goal is to use the maximum principle on the evolution equation for $\beta$ and so we need to estimate the partial derivatives of $F$ as follows
\begin{align} 
\frac{\partial F}{\partial u}&=\frac{-nv^2}{(n+y\tilde{\delta}^{ij} y_{ij})^2} = \frac{-n}{H^2}\le 0
\\\frac{\partial F}{\partial p_k}&=\frac{-2nyy_k -2y^2y_k\delta^{ij}y_{ij} -2 y^2 y^i y_{ik}}{(n+y\tilde{\delta}^{ij} y_{ij})^2}
\\\frac{\partial F}{\partial a_{lm}} &= \frac{yv^2}{(n+y\tilde{\delta}^{ij} y_{ij})^2} y\tilde{\delta}^{lm}  = \frac{y^2}{H^2}\tilde{\delta}^{lm} \ge 0
\end{align}
as well as the second partial derivatives
\begin{align} 
\frac{\partial^2 F}{\partial u^2}&= \frac{2n v^2\tilde{\delta}^{ij}y_{ij}}{(n+y\tilde{\delta}^{ij} y_{ij})^3}\ge 0
\\\frac{\partial^2 F}{\partial u\partial a_{lm}}&= \frac{2n y v^2}{(n+y\tilde{\delta}^{ij} y_{ij})^3} \tilde{\delta}^{lm}\ge 0
\\\frac{\partial^2 F}{\partial u \partial p_k}&=\frac{2n}{(n+y\tilde{\delta}^{ij} y_{ij})^3} \left (-ny_k-yy_k\delta^{ij}y_{ij} -2yy^ly_{kl} +\frac{yy_ky^iy^jy_{ij}}{v^2} \right )
\\\frac{\partial^2 F}{\partial p_k \partial p_n}&=\frac{-2y\delta_{kn}\left ( n+y\delta^{ij}y_{ij}\right )}{(n+y\tilde{\delta}^{ij} y_{ij})^2} + \frac{-2y^2y_{nk}}{(n+y\tilde{\delta}^{ij} y_{ij})^2}
\\&+4\frac{ny^2y_k +y^3y_k\delta^{ij}y_{ij} + y^3 y^i y_{ik}}{(n+y\tilde{\delta}^{ij} y_{ij})^3}\left (\frac{2y^iy_{in}}{v^2}-\frac{2y^iy^jy_ny_{ij}}{v^4} \right )
\\\frac{\partial^2 F}{\partial a_{lm} \partial p_k }&=\frac{2y^2y_k}{(n+y\tilde{\delta}^{ij} y_{ij})^2}\tilde{\delta}^{lm}+\frac{2y^2v^2}{(n+y\tilde{\delta}^{ij} y_{ij})^2}\left (\frac{-y_l\delta_{lk}-y_m\delta_{mk}}{v^2} + \frac{2y_ly_my_k}{v^4} \right )\tilde{\delta}^{lm}
\\&+\frac{2y^3v^2}{(n+y\tilde{\delta}^{ij} y_{ij})^3}\left(\frac{2y^iy_{ik}}{v^2}-\frac{2y^iy^jy_ky_{ij}}{v^4} \right )\tilde{\delta}^{lm}
\\\frac{\partial^2 F}{\partial a_{lm} \partial a_{no}}&= \frac{-2y^3 v^2}{(n+y\tilde{\delta}^{ij} y_{ij})^3} \tilde{\delta}^{lm}\tilde{\delta}^{no}\le 0
\end{align}
where the inequalities should be understood as communicating positive or negative symmetric matrices. Now we try to deal with some problematic terms
\begin{align}
\frac{\partial^2 F}{\partial p_k\partial p_n} y_{ik}y_{jn}y^{ij} &\le-C_4 e^{-t/n}\beta^{3/2} -C_5 e^{-2t/n}\beta^2+C_6e^{-5t/n}\beta^{5/2}
\\\frac{\partial^2 F}{\partial u \partial p_k} y_{ik}y_jy^{ij}&\le 0
\\\frac{\partial^2 F}{\partial u\partial a_{lm}}y_{ilm}y_jy^{ij}&\le C_7e^{-2t/n}|\nabla^3 y|\beta^{1/2}
\\ \frac{\partial^2 F}{\partial a_{lm} \partial p_k} y_{ilm}y_{jk}y^{ij}&\le C_{8}e^{-3t/n}|\nabla^3y| \beta+C_{9}e^{-4t/n}|\nabla^3y| \beta^{3/2}
\end{align}
for $t$ large enough.

Now by applying our previous estimates for $y, H$ and $v$ we find the evolution inequality
\begin{align}
&\beta_t\le 2\frac{\partial F}{\partial a_{lm}} \beta_{lm}+2\frac{\partial F}{\partial p_k}\beta_k -C_1 \beta-C_2 e^{-2t/n}|\nabla^3y|^2+C_3e^{-2t/n}\beta  
\\& -C_4 e^{-t/n}\beta^{3/2}-C_5 e^{-2t/n}\beta^2+C_6e^{-5t/n}\beta^{5/2}
\\&+ C_{7}e^{-2t/n}|\nabla^3y|\beta^{1/2} + C_{8}e^{-3t/n}|\nabla^3y| \beta+C_{9}e^{-4t/n}|\nabla^3y| \beta^{3/2}
\end{align} 
for $t$ large enough. We first note that the $C_3$ term can be handled by the $C_1$ term for $t$ large enough.

Now we would like to use Corollary \ref{CrudeC^2Bound} to trade some $\beta$ terms in for growth bounds in order to control the evolution equation for $\beta$. To this end we notice that $C_6e^{-5t/n}\beta^{5/2} \le C e^{-4t/n}\beta^2$ which can be controlled by the $C_5$ term. Now we use Young's inequality to break up the $C_{7}$, $C_{8}$ and $C_{9}$ terms into controllable pieces
\begin{align}
C_{7}e^{-2t/n}|\nabla^3y|\beta^{1/2} &\le C_{7} \left(\frac{1}{2}e^{-3t/n}|\nabla^3y|^2 + \frac{1}{2}e^{-t/n}\beta \right)
\end{align}
where the first term  is controlled by the $C_2$ term and the second term is controlled by the $C_1$ term. Similarly,
\begin{align}
C_{8}e^{-3t/n}|\nabla^3y| \beta &\le C_{8} \left(\frac{1}{2}e^{-3t/n}|\nabla^3y|^2 + \frac{1}{2}e^{-3t/n}\beta^2 \right)
\\C_{9}e^{-4t/n}|\nabla^3y| \beta^{3/2}&\le C_{9} \left(\frac{1}{2}e^{-3t/n}|\nabla^3y|^2 + \frac{1}{2}e^{-5t/n}\beta^3 \right)
\\&\le C_{9} \left(\frac{1}{2}e^{-3t/n}|\nabla^3y|^2 + \frac{1}{2}e^{-3t/n}\beta^2 \right)
\end{align}
where the first term in each line is controlled by the $C_2$ term and the second term is controlled by the $C_5$ term. 

Now we arrive at a simple evolution inequality for $\beta$, for large enough $t$
\begin{align}
\beta_t&\le 2\frac{\partial F}{\partial a_{lm}} \beta_{lm}+2\frac{\partial F}{\partial p_k}\beta_k -C_1 \beta
\end{align} 
to which the result follows by the ODE maximum principle at infinity.
\end{proof}

\begin{Cor}\label{CrudeTFABound}
For hypersurfaces satisfying the hypotheses of Theorem \eqref{LTE} the corresponding solution to IMCF in hyperbolic space satisfies
\begin{align}
|A_{ij} - g_{ij}|\le C e^{-\left(\gamma+\frac{1}{n} \right ) t}
\end{align}
where $\gamma > 0$ and $(x,t) \in \R^n\times [0,\infty)$.
\end{Cor}
\begin{proof}
Since $\Sigma_t$ is a graph over $\R^n$ we can write $A_{ij} = \frac{1}{y v}\nabla^0_i\nabla^0_j y +g_{ij}$ and so we find that $|A_{ij}- g_{ij}| \le C \frac{y}{v}|\nabla^0 \nabla^0 y|\le Ce^{\left(-\gamma t-\frac{1}{n}t \right ) }$, as desired.
\end{proof}

The last asymptotic estimate we would like is to improve Corollary \ref{CrudeTFABound} so that $|A_{ij}-g_{ij}| \le C e^{-2t/n}$, which is the optimal decay rate we expect for IMCF in Hyperbolic space.

\begin{Thm}\label{AA}
For hypersurfaces satisfying the hypotheses of Theorem \eqref{LTE} the corresponding solution to IMCF in hyperbolic space satisfies
\begin{align}
|A_{ij}-g_{ij}| \le C e^{-2t/n}
\end{align}
\end{Thm}
\begin{proof}
If we define $G = |A_{ij}-g_{ij}|^2=|A|^2 - 2H +n$ then we can find the following evolution inequality for $G$
\begin{align}
\left (\partial _t - \frac{1}{H^2} \Delta \right )G&= \left (\partial _t - \frac{1}{H^2} \Delta \right )|A|^2 -2 \left (\partial _t - \frac{1}{H^2} \Delta \right ) H
\\&= -\frac{4}{H^3}A(\nabla H,\nabla H) -\frac{2}{H^2}|\nabla A|^2 + 2\frac{n + |A|^2}{H^2} |A|^2
\\& - \frac{4}{H}A^3 + \frac{4}{H^3}|\nabla H|^2 + 2\frac{|A|^2}{H} -2\frac{n}{H}
\\&= -\frac{4}{H^3}(A_i^j - \delta_i^j)\nabla_iH\nabla^jH-\frac{2}{H^2}|\nabla A|^2 -\frac{4n}{H^2} G
\\&+\frac{6n|A|^2}{H^2} + \frac{2|A|^2}{H^2} -\frac{10n}{H} + \frac{2|A|^2}{H} - \frac{4}{H} A^3 + \frac{4n^2}{H}
\\&\le -\frac{4n}{H^2} G +  \frac{|\nabla H|^2}{H^2}\left ( 4\frac{|A - g|}{H} - \frac{3}{n+2} \right )
\\& +\frac{2}{H^2} (|A|^4-HA^3) +\frac{4n}{H}\left (1- \frac{n}{H} \right ) 
\\& + \frac{6n}{H} \left (\frac{|A|^2}{H} - 1\right ) + \frac{2}{H}(|A|^2 - A^3)
\end{align}
where we have used the fact that $|\nabla A|^2 \ge \frac{3}{n+2}|\nabla H|^2$. From the asymptotic estimates we already have in Theorem \ref{barriers}, \ref{globalHcontrol} and Lemmas \ref{AG}, \ref{L3} and Corollary \ref{CrudeTFABound} we find
\begin{align} 
\left (\partial _t - \frac{1}{H^2} \Delta \right )G&\le -\frac{4n}{H^2} G + C_1e^{-\left (\frac{4}{n} + \gamma \right )t} + C_2 e^{-\left (\frac{3}{n} + 3\gamma \right )t}
\end{align}
Using previous estimates from this paper we find
\begin{align}
\left (\partial _t - \frac{1}{H^2} \Delta \right )G&\le -\frac{4}{n +C e^{-2t/n}} G + C_1e^{-\left (\frac{4}{n} + \gamma \right )t} + C_2 e^{-\left (\frac{3}{n} + 3\gamma \right )t}
\end{align}

Now we can use an integrating factor to rewrite
\begin{align}
\frac{d }{dt}\left ( (1+ne^{2t/n})^2G_{sup} \right ) \le (1+ne^{2t/n})^2 \left ( C_1e^{-\left (\frac{4}{n} + \gamma \right )t} + C_2 e^{-\left (\frac{3}{n} + 3\gamma \right )t} \right )
\end{align}
which implies, by integrating, that $G_{sup}(t) \le C e^{-\frac{4t}{n}} +  Ce^{-\left( \frac{3}{n}+3\gamma\right ) t} $. Then by applying Theorem \eqref{ODEmax} we get the estimate $G(x,t) \le C e^{-\frac{4t}{n}} +  Ce^{-\left( \frac{3}{n}+3\gamma\right )t}$. 

Now if $3 \gamma \ge \frac{1}{n}$ then we are done so if it is not, $3 \gamma < \frac{1}{n}$, then we can recalculate the evolution inequality (2) with the new bound on $G$ to find

\begin{align} 
\left (\partial _t - \frac{1}{H^2} \Delta \right )G&\le -\frac{4n}{H^2} G + C_1e^{-\left (\frac{6}{n} + \frac{3\gamma}{2} \right )t} + C_2 e^{-\left (\frac{9}{2n} + \frac{9}{2}\gamma \right )t}
\end{align}

Then using the same analysis as above we would find $G(x,t) \le C e^{-\frac{4t}{n}}$ since, when we integrate the right hand side of (3), all the terms will be negative and hence can be thrown out except for the constant which is then multiplied by the integrating factor yielding the correct asymptotic decay rate.

\end{proof}

\section{Conclusion}
\label{sec:3}

In this paper we have seen the utility of the ODE maximum principle at infinity by using Theorem \eqref{ODEmax} to prove a new long time existence theorem and asymptotic analysis for non-compact solutions of IMCF in hyperbolic space, Theorem \eqref{LTE}. We fully expect the ODE maximum principle at infinity to be useful to many more results in the study of non-compact solutions of any geometric evolution equation, especially when it is hard to control terms appearing in an evolution equation on the whole domain as in Theorem \eqref{barriers}.


\end{document}